\documentclass[12pt]{amsart}
\usepackage{amssymb}
\usepackage{amscd}
\usepackage{amsmath}
\usepackage{amsmath}
\usepackage[all]{xy}
\usepackage{mathrsfs}

\usepackage{color}

\theoremstyle{plain}
\newtheorem{theorem}{Theorem}[section]
\newtheorem{lemma}[theorem]{Lemma}
\newtheorem{corollary}[theorem]{Corollary}
\newtheorem{proposition}[theorem]{Proposition}

\newtheorem{conjecture}[theorem]{Conjecture}
\theoremstyle{definition}

\newtheorem{example}[theorem]{Example}

\newtheorem{hypothesis}[theorem]{Hypothesis}
\theoremstyle{definition}
\newtheorem{definition}[theorem]{Definition}
\newtheorem{remark}[theorem]{Remark}

  1

\DeclareMathOperator{\Aut}{Aut}

\DeclareMathOperator{\Gal}{Gal}

\DeclareMathOperator{\Hom}{Hom}

\DeclareMathOperator{\Spec}{Spec}
\DeclareMathOperator{\N}{N}

\DeclareMathOperator{\coker}{coker}

\newcommand{\CC}{\mathbb{C}}

\newcommand{\GG}{\mathbb{G}}

\newcommand{\QQ}{\mathbb{Q}}

\newcommand{\RR}{\mathbb{R}}

\newcommand{\ZZ}{\mathbb{Z}}

\newcommand{\co}{\mathcal{O}}

\newcommand{\bz}{\mathbb{Z}}

\newcommand{\ord}{\mathrm{ord}}

\begin{document}

%%%%%%%%%%%%%%%%%%%%%%%%%%%%%%%
%%% TITLE, AUTHOR, ABSTRACT %%%
%%%%%%%%%%%%%%%%%%%%%%%%%%%%%%%
\title[]{On Stark elements of arbitrary weight\\ and their $p$-adic families
%On the arithmetic of generalized Stark elements
}

\author{David Burns, Masato Kurihara and Takamichi Sano}
%\thanks{Preliminary version of June, 2015}

\begin{abstract}
 We develop a detailed arithmetic theory related to special values at arbitrary integers of the Artin $L$-series of linear characters. To do so we define canonical generalized Stark elements of arbitrary `rank' and `weight', thereby
extending the classical theory of Rubin-Stark elements. We then formulate an extension to arbitrary weight of the refined version of the Rubin-Stark Conjecture that we studied in an earlier article and also show that generalized Stark elements
constitute a $p$-adic family by formulating precise conjectural congruence relations between elements of differing weights. We prove both of these  conjectures in several important cases.
\end{abstract}

\address{King's College London,
Department of Mathematics,
London WC2R 2LS,
U.K.}
\email{david.burns@kcl.ac.uk}

\address{Keio University,
Department of Mathematics,
3-14-1 Hiyoshi\\Kohoku-ku\\Yokohama\\223-8522,
Japan}
\email{kurihara@math.keio.ac.jp}

\address{Osaka University,
Department of Mathematics,
1-1 Machikaneyama-cho\\
Toyonaka\\Osaka\\560-0043,
Japan}
\email{tkmc310@gmail.com}

\maketitle

%\section{Introduction} \label{Intro}%This is an updated version of the submitted file dburns.tex
%\tableofcontents
%

\section{Introduction} \label{Intro}

%The principal conjecture of Stark asserts that the leading term at
%$s=0$ of Artin $L$-functions is equal, to within an undetermined
%algebraic factor, to a regulator constructed from algebraic units.

\subsection{}The seminal conjecture of
Stark predicts that canonical elements constructed (unconditionally) from the leading terms at zero of the Artin $L$-series of complex linear characters should belong to
 the rational vector spaces that are spanned by the $r$-th exterior powers of suitable groups of algebraic units, where $r$ denotes the order of vanishing at zero of the
  relevant $L$-series.

It is believed that if Stark's conjecture is true, then these `Stark elements' should constitute a higher rank Euler system for  the multiplicative group $\mathbb{G}_m$ over number fields and so there is considerable interest in studying their detailed arithmetic properties, especially the integrality. 

The basic integral properties of Stark elements were first studied by Stark himself in \cite{stark4}, and then by Tate in \cite{tatebook}, for the case  $r=1$ and then subsequently by Rubin in  \cite{rubinstark} where the so-called `Rubin-Stark Conjecture' was formulated in the setting of general order of vanishing.

More recently, we formulated a strong refinement of the Rubin-Stark Conjecture in \cite[Conj. 7.3]{bks1}, obtained a natural interpretation of this conjecture in terms of a general theory of `arithmetic zeta elements' that was motivated by an earlier approach of Kato to the formulation of generalized Iwasawa main conjectures and derived a series of consequences of our conjecture concerning the detailed algebraic properties of number fields.

%roerties of ertain consequences concerning xplored its consequences  for the detailed we developed an integral theory which refines
%Rubin's theory, and which describes the arithmetic significance of the
%Stark elements for $s=0$ (see ).
%We also generalize this theory in \cite{bks1} to any integer values $s=j=-w/2$
%in this paper.

In this article we shall use the leading terms at arbitrary integer points of the $L$-series of linear characters to unconditionally define for non-negative integers $r$ and even integers $w$ canonical `(generalized) Stark elements of rank $r$ and weight $w$'. In this context the `rank' relates to the exterior power of the arithmetic module in which the element is constructed and the `weight' to the integer point $j_w := -w/2$ at which one takes the leading term of the $L$-series (so that $w$ is the weight of the associated motive $h^0(\Spec L)(j_w)$). In particular, in weight $0$, our construction recovers the classical theory of Rubin-Stark elements, and hence in weight $0$ and rank $1$ recovers the original constructions of Stark.

Our approach will show that generalized Stark elements of any fixed rank and weight should encode (in a very explicit way) detailed arithmetic information concerning the Galois structure of important \'etale cohomology groups.

In addition, our approach leads naturally to the simultaneous study of generalized Stark elements of differing weights and thereby introduces the perspective of $p$-adic families to the investigation. We remark that this philosophy of $p$-adic families has not hitherto been used in the setting of the Rubin-Stark Conjecture and is we feel worthy of further consideration.

To be a little more precise about the results proved here we fix an odd prime $p$ and a finite abelian extension $L/K$ of number fields. We assume that $L/K$ is
 unramified outside a finite set of places $S$ of $K$ containing all places that are either archimedean or $p$-adic and we set $G := {\rm Gal}(L/K)$.

Then, as a first step, in Conjecture \ref{higherfitt} we predict that
generalized Stark elements of weight $w$ (and of appropriate rank) over $L$  explicitly determine the initial Fitting ideal of the \'etale cohomology group $H^2(\mathcal{O}_{L,S},\ZZ_p(1-j_w))$, regarded as a $\ZZ_p[G]$-module in the natural way.

This conjecture constitutes a natural extension of \cite[Conj. 7.3]{bks1} from the case of weight $0$ to the case of arbitrary weight. In addition, an   interpretation of generalized Stark elements in terms of the theory of arithmetic zeta elements allows us to prove Conjecture \ref{higherfitt} for all absolutely abelian fields and for the minus part of CM-extensions of totally real fields (see Theorems \ref{evidence} and \ref{thm1}).

Next we write $p^{n}$ for the number of $p$-power roots of unity in $L$ and note that the Galois modules $H^2(\mathcal{O}_{L,S},\ZZ/p^{n}(1-j))(j-k)$ and $H^2(\mathcal{O}_{L,S},\ZZ/p^{n}(1-k))$ are isomorphic for any choice of integers $j$ and $k$. In particular, since Conjecture \ref{higherfitt} implies that Stark elements of weight $w$ over $L$ determine the initial Fitting ideal of the $(\ZZ/p^n\ZZ)[G]$-module $H^2(\mathcal{O}_{L,S},\ZZ/p^{n}(1-j_w))$, it suggests Stark elements of differing weights (and fixed rank) should be related by congruences modulo $p^{n}$.

In Conjecture \ref{congruence conjecture} we use algebraic techniques developed in \cite{bks1} to formulate, modulo the assumed validity of a weak version of Conjecture \ref{higherfitt}, a precise and explicit family of congruence relations between Stark elements of differing weights over arbitrary number fields $L$.

We show that this very general family of conjectural congruences recovers upon appropriate specialization a wide variety of results in the literature ranging from the classical congruences of Kummer concerning Bernoulli numbers to the results of Beilinson and H\"uber-Wildeshaus concerning the cyclotomic elements of Deligne-Soul\'e and a more recent conjecture of Solomon concerning certain `explicit reciprocity laws' for Rubin-Stark elements. These various connections allow us, in particular, to derive strong evidence, both theoretical and numerical, in support of Conjecture \ref{congruence conjecture} (see Theorem \ref{cong evidence} and Remark \ref{cong evid rem}).

%Thus our conjectures on the generalized Stark elements $\eta(j)$ give a new picture of
%{\it $p$-adic families} of algebraic elements for different special values of $L$ series, which
%we think is a very important phenomenon to study.

The main contents of this article is as follows. In \S\ref{section2} we give the definition of generalized Stark elements and then in \S\ref{statement of conjs} we formulate our central conjectures concerning these elements (in Conjectures \ref{higherfitt} and \ref{congruence conjecture}). In \S\ref{zeta evidence}  we give a natural interpretation of generalized Stark elements in terms of the theory of zeta elements and then use this interpretation to prove Conjecture \ref{higherfitt} in some important cases. Finally, in \S\ref{cong evi proof}, we relate special cases of Conjecture \ref{congruence conjecture} to well-known results in the literature and thereby deduce some supporting evidence for it.

\subsection{}\label{notation} For the reader's convenience we end the introduction by collecting together details concerning notation and conventions that are used in the sequel.

\subsubsection{Algebra} Let $E$ be a field of characteristic $0$. For any abelian group $A$, we denote $E\otimes_\ZZ A$ by $EA$. If $A$ is a $\QQ$-vector space, we sometimes denote $E\otimes_\QQ A$ also by $EA$. Similarly, if $E$ is an extension of $\QQ_p$ ($p$ is a prime number) and $A$ is a $\ZZ_p$-module, we denote $E\otimes_{\ZZ_p}A$ and $E\otimes_{\QQ_p} A$ also by $EA$.
 For any integer $m$, we denote $A/mA$ simply by $A/m$.
 %If $A$ is a $\ZZ_p$-module, we denote the Pontryagin dual $\Hom_{\ZZ_p}(A,\QQ_p/\ZZ_p)$ by $A^\vee$.

For a commutative ring $R$ and an $R$-module $M$ we set
$M^\ast := \Hom_R(M,R)$.
%\textcolor{red}{This notation is in principal ambiguous
%since $M$ can simultaneously be regarded as a module over a different ring
%(such as $\ZZ$). However,} \textcolor{green}{(I feel this explanation
%can be removed.)}
%In this article we mainly consider rings of the form $R[G]$
 %for a commutative ring $R$ (such as $\ZZ, \QQ, \ZZ_p, \QQ_p, \CC_p$ or $\ZZ/p^n$) and abelian group $G$ and for any $R[G]$-module $M$
% there is a canonical
 % bijection $\textcolor{blue}{M^\ast =}\Hom_R(M,R) \simeq \Hom_{R[G]}(M,R[G])$ that sends each $f$ to $\sum_{\sigma \in G}f(\sigma(\cdot))\sigma^{-1}.$
%So in many case the notation $M^\ast$ does not make confusion.
%But we endow a $G$-action on $\Hom_R(M,R)$ by
%$$\sigma \cdot f :=f\circ \sigma^{-1},$$
%so note that the above bijection does not give an isomorphism of $R[G]$-modules.
If $M$ is a free $R$-module with basis $\{b_1,\ldots,b_r\}$, then for each $i$ with $1\leq i \leq r$ we write $b_i^\ast$ for the homomorphism $M\to R$ that sends $b_j$ to $1$ if $i=j$ and to $0$ if $i\neq j$.

For any field $E$, the absolute Galois groups is denoted by $G_E$. Let $c\in G_\RR$ denote the complex conjugation. For a $\ZZ[G_\RR]$-module $M$, let $M^\pm$ be the submodule $\{ a \in M \mid c\cdot a =\pm a\}$ of $M$. We also use the idempotent $e^\pm :=\frac{1\pm c}{2}$ of $\ZZ[\frac12][G_\RR]$ and the decomposition $M=M^+\oplus M^-$ with $M^\pm=e^\pm M$ for any $\ZZ[\frac12][G_\RR]$-module $M$.

\subsubsection{Arithmetic}

Fix an algebraic closure $\overline \QQ $ of $\QQ$. For any non-negative integer $m$, we denote by $\mu_{m}$ the subgroup of all $m$-th roots of unity in $\overline \QQ^\times$. As usual, we denote $\mu_{p^n}$ ($p$ is a prime number) by $\ZZ/p^n(1)$, and $\varprojlim_n \mu_{p^n}$ by $\ZZ_p(1)$. For any integer $j$, $\ZZ_p(j)$ and $\QQ_p(j)$ are defined in the usual way.

For a number field $K$, i.e. a finite extension of $\QQ$ in $\overline \QQ$, we write $S_\infty(K)$, $S_\CC(K)$ and $S_p(K)$ for the set of archimedean,
complex and $p$-adic places of $K$ respectively. We write $S_\infty$ for $S_\infty(K)$ if there is no danger of confusion. The ring of integers of $K$ is denoted by $\co_K$. For a finite set $S$ of places of $K$, the ring of $S$-integers of $K$ is denoted by $\co_{K,S}$. If $L$ is a finite extension of $K$, then the set of places of $K$ which ramify in $L$ is denoted by $S_{\rm ram}(L/K)$ and the set of places of $L$
lying above any given set of places $S$ of $K$ is denoted by $S_L$. The ring of $S_L$-integers of $L$ is denoted by $\co_{L,S}$ instead of $\co_{L,S_L}$.

%Concerning notation on the $p$-adic \'etale cohomology, we use the standard notation, as in \cite{BFetnc} for example.

Let $L/K$ be a finite abelian extension with Galois group $G$. Let $S$ and $T$ be finite disjoint sets of places of $K$ such that $S_\infty(K) \cup S_{\rm ram}(L/K) \subset S$.
%Then, the $S$-truncated equivariant $L$-function for $K/k$ is defined by
%$$\theta_{K/k,S}(s):=\prod_{v \notin S}(1-{\rm Fr}_v^{-1}{\N}v^{-s})^{-1},$$
Then, for a character $\chi \in \widehat G:=\Hom_\ZZ(G,\CC^\times)$, the $S$-truncated $T$-modified $L$-function is defined by
$$L_{K,S,T}(\chi,s):=\prod_{v\in T}(1-\chi({\rm Fr}_v) {\N}v^{1-s})\prod_{v\notin S}(1- \chi({\rm Fr}_v){\N}v^{-s})^{-1} \ ({\rm Re}(s)>1)$$
where ${\rm Fr}_v \in G$ is the Frobenius automorphism at a place of $L$ above $v$, and ${\N}v$ is the cardinality of the residue field $\kappa(v)$ of $v$. The function
$L_{K,S,T}(\chi,s)$ continues meromorphically to the whole complex plane and its leading term at an integer $j$ is denoted by
$L_{K,S,T}^\ast(\chi,j)$. The $S$-truncated $T$-modified $L$-function for $L/K$ is defined by setting
$$\theta_{L/K,S,T}(s):=\sum_{\chi \in \widehat G}L_{K,S,T}(\chi^{-1},s)e_\chi \,\,\,\text{ with } \,\,\, e_\chi:=\frac{1}{\# G}\sum_{\sigma \in G}\chi(\sigma) \sigma^{-1}$$
and has leading term at $s=j$ equal to $\theta_{L/K,S,T}^\ast(j):=\sum_{\chi \in \widehat G}L_{K,S,T}^\ast(\chi^{-1},j)e_\chi \in \CC[G]^\times$. When $T=\emptyset$, we omit it from notations (so we denote $L_{K,S,\emptyset}(\chi,s)$ by $L_{K,S}(\chi,s)$, for example). Note that
$$\theta_{L/K,S,T}(s)= \delta_{L/K,T}(s)\cdot \theta_{L/K,S}(s) $$
with $\delta_{L/K,T}(s) := \prod_{v \in T}(1-{\N}v^{1-s}{\rm Fr}_v^{-1})$.
%($\delta_{L/K,\emptyset}(s)$ is understood to be $1$.) The leading term of $\theta_{L/K,S,T}(s)$ at $s=j$ is denoted by $\theta_{L/K,S,T}^\ast(j) \in \CC_p[G]^\times$. Note that
%$$\varepsilon_j \theta_{L/K,S,T}^\ast(j)=\varepsilon_j \cdot \delta_{L/K,T}(j) \cdot \theta_{L/K,S}^\ast(j).$$

\section{Generalized Stark elements} \label{section2}

\subsection{The general set up}\label{setup}

Let $L/K$ be a finite abelian extension of number fields. Set $G:=\Gal(L/K)$. Fix an {\it odd} prime number $p$. For each place $w$ of $L$,
we fix an algebraic closure $\overline L_w$ of $L$ and an embedding $\overline \QQ \hookrightarrow \overline L_w$. From this, we regard $G_{L_w}$ as a subgroup of $G_L$,
and the localization map of Galois cohomology $H^i(L, \cdot)\to H^i(L_w,\cdot)$ is defined by the restriction map. Also, for each
place $w$ in $S_\infty(L)$, we identify $\overline L_w$ with $\CC$.
For each integer $j$ we set
\[ S_{\infty}^j(L) := \begin{cases} S_\infty(L) &\text{if $j$ is even,}\\
                                    S_\CC(L) &\text{if $j$ is odd,}\end{cases} \]
and note that

\begin{equation*}\label{first decomp} Y_L(j):=\bigoplus_{w\in S_\infty(L)} H^0(L_w,\ZZ_p(j))= \bigoplus_{w \in S^j_\infty(L)}\ZZ_p(j).\end{equation*}

In particular, setting $\xi:=(e^{2\pi \sqrt{-1}/p^n})_n \in \ZZ_p(1)$ one obtains a $\ZZ_p$-basis $\{w(j)\}_{w \in S^j_\infty(L)}$ of $Y_L(j)$,
which is defined by $w(j)=(w(j)_{w'})_{w'}$ where
$$w(j)_{w'}:=\begin{cases}
\xi^{\otimes j} &\text{if $w'=w$,}\\
0 &\text{if $w'\neq w.$}
\end{cases}
$$

Next we note that the complex conjugation $c$ in $G_\RR$ acts on the Betti cohomology
$$H_L(j):=H_B^0(\Spec L(\CC),\QQ(j)) =\bigoplus_{\iota: L\hookrightarrow \CC}(2\pi\sqrt{-1})^j \QQ$$
by $c\cdot (a_\iota)_\iota:=(c\cdot a_\iota)_{c\circ \iota}$ for each $a_\iota$ in
$(2\pi\sqrt{-1})^j\QQ$ and we set
$$H_L(j)^+:=e^+ H_L(j) \,\,\,\text{ with } \,\,\, e^+:=\frac{1+c}{2}.$$
Note that the natural decomposition $\CC= \RR \oplus \RR(-1)$ induces an isomorphism
\begin{equation} \label{period map}
 \RR \otimes_{\QQ}L
 % \prod_{L \hookrightarrow \CC_p}\CC_p
 %=\CC_p H_L(j)=\CC_pH_L(j)^+ \oplus \CC_p H_L(j)^-
 \simeq  \RR H_L(j)^+ \oplus \RR H_L(j-1)^+ .
 \end{equation}

For each embedding $\iota': L \hookrightarrow \CC$ we define $\iota'_{j}=(\iota'_{j,\iota})_\iota$ in $H_L(j)$ by setting
$$\iota'_{j,\iota}:=\begin{cases}
(2\pi\sqrt{-1})^j &\text{if $\iota=\iota'$,}\\
0 &\text{if $\iota\neq \iota'.$}
\end{cases}
$$

Then, if for each place $w$ in $S_\infty(L)$ we write $\iota_w: L \to \CC$ for the embedding induced by the fixed embedding
$\overline \QQ \hookrightarrow \overline L_w=\CC$, we obtain an isomorphism of $\QQ_p[G]$-modules
\begin{eqnarray}
%\beta:
 Y_L(j) \otimes_{\ZZ_p}\QQ_p \stackrel{\sim}{\to} H_L(j)^+ \otimes_\QQ \QQ_p \label{beta}
\end{eqnarray}
that sends each element $w(j)$ to $e^+\iota_{w,j}.$

For each place $v$ in $S_\infty(K)$ we now fix a place $w_v$ in $S_\infty(L)$ that lies above $v$ and set $S_\infty(L)/G := \{w_v \mid v \in S_\infty(K) \}$. For any idempotent $\varepsilon$ in $\ZZ_p[G]$ we set
$$W_j^\varepsilon:=\{ w \in S_\infty^j(L)\cap (S_\infty(L)/G) \mid \varepsilon \cdot w(-j) \neq 0 \}$$
and then define $r_j^\varepsilon:=\# W_j^\varepsilon$.

\begin{lemma} \label{lemma idempotent}
If $\varepsilon$ is a primitive idempotent of $\ZZ_p[G]$, then $\varepsilon Y_L(-j)$ is a free $\ZZ_p[G]\varepsilon$-module of rank $r_j^\varepsilon$ with basis $\{ \varepsilon \cdot w(-j) \mid w \in W_j^\varepsilon\}$.
\end{lemma}

\begin{proof}
For each $w \in S_\infty^j(L)$, the $\ZZ_p[G]$-submodule of $Y_L(-j)$ generated by $w(-j)$ is projective, since $p$ is odd. Thus, if $\varepsilon$ is a primitive idempotent  (so that the ring $\ZZ_p[G]\varepsilon$ is local), then $ \ZZ_p[G]\varepsilon \cdot w(-j)$ is either zero or a free $\ZZ_p[G]\varepsilon$-module of rank one and so the decomposition
\[ \varepsilon  Y_L(-j) = \bigoplus_{w\in S_\infty^j(L)\cap S_\infty(L)/G}  \ZZ_p[G] \varepsilon\cdot w(-j)\]
implies that $\varepsilon Y_L(-j)$ is free with basis $\{ \varepsilon \cdot w(-j) \mid w \in W_j^\varepsilon\}$.
\end{proof}

The algebra $\ZZ_p[G]$ is semilocal and so every idempotent of $\ZZ_p[G]$ is
a sum of primitive idempotents. By Lemma \ref{lemma idempotent},
we may consider, without any loss of generality,
an idempotent satisfying the following condition.

\begin{hypothesis}\label{hypo}
$\varepsilon$ is an idempotent of $\ZZ_p[G]$ such that
the $ \ZZ_p[G]\varepsilon$-module $\varepsilon Y_L(-j)$ is free of rank $r_j^\varepsilon$ and has as basis the set  $\{\varepsilon \cdot w(-j) \mid w \in W_j^\varepsilon\}$.\end{hypothesis}

%We remake that, since $\ZZ_p[G]$ is semilocal, there always exists an idempotent $\varepsilon \in \ZZ_p[G]$ which satisfies ($\ast$). Indeed, if $ \ZZ_p[G]\varepsilon$ is local, then  for each $w \in S_\infty^j(L)$ we see that $ \ZZ_p[G]\varepsilon \cdot w(-j)$ is either zero or free of rank one, so the decomposition
%\[ \varepsilon  Y_L(-j) = \bigoplus_{w\in S_\infty^j(L)\cap S_\infty(L)/G}  \ZZ_p[G] \varepsilon\cdot w(-j)\]
%implies that $\varepsilon Y_L(-j)$ is free with basis $\{ \varepsilon \cdot w(-j) \mid w \in W_j^\varepsilon\}$.

%We also fix a primitive idempotent $\varepsilon$ of $\ZZ_p[G]$, i.e. an idempotent such that $\varepsilon \ZZ_p[G]$ is local. (Such an idempotent always exists since $\ZZ_p[G]$ is semilocal.)

%\begin{lemma}\label{free module lemma} The $\varepsilon \ZZ_p[G]$-module $\varepsilon Y_L(j)$ is free with basis
%the set $W^{\varepsilon}_j$ of non-zero elements of the form $\varepsilon\cdot w(j)$ for $w$ in $S_\infty^j(L)\cap S_\infty(L)/G$.\end{lemma}

%\begin{proof}

%The decomposition (\ref{first decomp}) induces a decomposition of $\ZZ_p[G]$-modules
%
%\[ \varepsilon  Y_L(j) = \bigoplus_{w\in S_\infty^j(L)\cap S_\infty(L)/G} \varepsilon \ZZ_p[G] \cdot w(j).\]
%
%This implies the claim since $\varepsilon$ is primitive so the ring $\varepsilon\ZZ_p[G]$ is local and hence each module $\varepsilon \ZZ_p[G] \cdot w(j)$ is either zero or
%free of rank one. \end{proof}

%In the sequel we write $r_j^\varepsilon$ for the rank $\# W^{\varepsilon}_j$ of the free $ \ZZ_p[G]\varepsilon$-module $\varepsilon Y_L(-j)$.

\begin{example} \label{example} Suppose $K$ is totally real and $L$ is CM and write $c$ for the complex conjugation in $G$.
 For each integer $j$ we obtain idempotents of $\ZZ_p[G]$ by setting $e_j^\pm := (1\pm (-1)^{j}c)/2$ and we abbreviate
 $W_j^{e_j^\pm}$ to $W_j^\pm$. Then Hypothesis \ref{hypo} is satisfied in each of the following cases.
\begin{itemize}
\item[(i)] If $\varepsilon = e_j^+$, then $\varepsilon\cdot w(-j) = w(-j)$ for each $w$ in $S_\infty^j(L)\cap S_\infty(L)/G = S_\infty(L)/G$ and so we have $W_j^+=S_\infty(L)/G$ and $r_j^\varepsilon =
\# S_\infty(L)/G = \# S_\infty(K) = [K:\QQ]$.
\item[(ii)] If $\varepsilon = e_j^-$, then $\varepsilon\cdot  w(-j) = 0$ for each $w$ in $S_\infty^j(L)$ so $W_j^-$ is empty and $r_j^\varepsilon =0$.
\end{itemize}

\end{example}

\subsection{The period-regulator isomorphisms}\label{period-regulator} In this section we assume the idempotent $\varepsilon$ satisfies Hypothesis \ref{hypo} with respect to $j$.

In the sequel we fix a finite set $S$ of places of $K$ which contains $S_\infty(K) \cup S_p(K) \cup S_{\rm ram}(L/K)$. We also fix (and do not explicitly mention) an isomorphism of fields $\CC \simeq \CC_p$.

We write $\widehat G^\varepsilon$ for the subset of $\widehat G$ comprising characters $\chi$ for which
$\varepsilon\cdot e_\chi \not= 0$.
We define a subset  of $\widehat G^\varepsilon$ by setting
\[ \widehat G_j^\varepsilon  := \{ \chi\in \widehat G^\varepsilon \mid {\rm dim}_{\CC_p}(e_\chi \CC_p H^1(\co_{L,S},\QQ_p(1-j))) = r_j^{\varepsilon} \}\]
and then obtain an idempotent of $ \QQ_p[G]\varepsilon$ by setting
\[ \varepsilon_j  := \sum_{\chi \in \widehat G_j^{\varepsilon}}e_\chi .\]
%where $\chi$ runs over the characters in $\widehat G_j^{\varepsilon}$.

%When $j \neq 1$, we also have $H_T^1(\co_{L,S},\QQ_p(1-j))=H^1(\co_{L,S},\QQ_p(1-j))$.

%It is straightforward to see that this element is .

%For a finite set $\Sigma$ of places of $K$ we define $Y_{L,\Sigma}$ to be the free $\ZZ_p$-module on the set $\Sigma_L$, i.e. $Y_{L,\Sigma}:=\bigoplus_{w \in \Sigma_L} \ZZ_p w$,
%and $X_{L,\Sigma}$ to be the kernel of
%$$Y_{L,\Sigma}  \to \ZZ_p; \ \sum_{w \in \Sigma_L} a_w w \mapsto \sum_{w\in \Sigma_L} a_w.$$
%We can naturally regard $Y_{L,\Sigma}$ and $X_{L,\Sigma}$ as $\ZZ_p[G]$-modules.
%Note that cohomologies $\bigoplus_{w \in T_L}H^i(\kappa(w),\ZZ_p)$ for $i=0,1$ are both isomorphic to $Y_{L,T}$.

\begin{remark}\label{exp idem} Lemma \ref{prop complex}(ii) below implies that for each
$\chi \in \widehat G^\varepsilon$ one has
\begin{eqnarray*}
 \chi\in \widehat G^\varepsilon_j \Longleftrightarrow
 %  e_\chi(H^0(\co_{L,S},\CC_p(1-j)) \oplus \bigoplus_{w\in T_L} H^0(\kappa(w),\CC_p(1-j)) \oplus H^2(\co_{L,S},\CC_p(1-j))) \text{ vanishes}\\
% &\Longleftrightarrow&
 \begin{cases}
 e_\chi \CC_p H^2(\co_{L,S},\QQ_p(1-j))
                 %\simeq e_\chi \CC_p X_{L,S\setminus S_\infty}
                  \,\,\, \text{vanishes}, &\text{if $j \neq 1$},\\
                % e_\chi \CC_p H^2(\co_{L,S},\QQ_p(1-j)) \,\,\,\text{vanishes}, &\text{otherwise,}
                 e_\chi ( \CC_p \oplus \CC_pH^2(\co_{L,S},\QQ_p)) \,\,\, \text{vanishes}, &\text{if $j=1$}.
\end{cases}
\end{eqnarray*}
%where for a set $\Sigma$ of places of $K$ we denote by $X_{L,\Sigma}$ the kernel of
%$$\bigoplus_{w \in \Sigma_L}\ZZ_p w  \to \ZZ_p; \ \sum_{w \in \Sigma_L} a_w w \mapsto \sum_{w\in \Sigma_L} a_w.$$
%and the unlabeled isomorphism is induced by class field theory.
By using this description one can deduce the following facts.
\begin{itemize}
\item[(i)] If $j<0$, then $\widehat G^\varepsilon_j = \widehat G^\varepsilon$ (by Soul\'e \cite[Th. 10.3.27]{NSW}) and so $\varepsilon_j = \varepsilon$.
\item[(ii)] If $j=0$, then
\begin{eqnarray*}
\widehat G_0^\varepsilon &=& \{\chi \in \widehat G^\varepsilon \mid e_\chi \CC_p X_{L,S\setminus S_\infty} =0\}\\
&=& \{\chi\in \widehat G^\varepsilon \mid \ord_{s=0}L_{K,S}(\chi,s) = r^\varepsilon_0 \}.
\end{eqnarray*}
Here (and in the sequel), for any finite set $\Sigma$ of places of $K$ we write $X_{L,\Sigma}$ for the kernel of the homomorphism $\bigoplus_{w \in \Sigma_L} \ZZ_p w  \to \ZZ_p$ sending each $w$ to $1$. The first displayed equality then follows by noting that, by class field theory, there is a canonical isomorphism $H^2(\co_{L,S},\QQ_p(1)) \simeq \QQ_p X_{L,S\setminus S_\infty}$, and the second equality follows directly from \cite[Chap. I, Prop. 3.4]{tatebook}.
\item[(iii)] If $j=1$, then Leopoldt's Conjecture for $L$ is equivalent to the vanishing of $H^2(\co_{L,S},\QQ_p)$ and hence implies that
%{\small
$$
\widehat G^\varepsilon_1  = \{ \chi \in \widehat G^\varepsilon \mid \chi \neq {\bf 1} \},
%& \text{ if }T=\emptyset, \\
%\{ \chi \in \widehat G^\varepsilon \mid e_\chi \CC_p \coker(H^1(\co_{L,S}, \QQ_p) \to \bigoplus_{w \in T_L}H^1(\kappa(w), \QQ_p))=0\} &\text{ if }T\neq \emptyset;
%\end{cases}
%&=& \{\chi \in \widehat G^\varepsilon \mid {\rm ord}_{s=1} L_{K,S,T}(\chi, s)=0  \};
%&=& \{ \chi \in \widehat G^\varepsilon \mid \chi \neq  {\bf 1} \text{ and }\chi({\rm Fr}_v)\neq 1 \text{ for all }v \in T \};
$$
where we write ${\bf 1}$ for the trivial character of $G$, and so $\varepsilon_1 =\varepsilon (1-e_{\bf 1})$.
\item[(iv)] If $j>1$, then Schneider's Conjecture \cite{ps} for $L$ is equivalent to the vanishing of $H^2(\co_{L,S},\QQ_p(1-j))$ and hence implies $\widehat G^\varepsilon_j = \widehat G^\varepsilon$ and so $\varepsilon_j = \varepsilon$.
\end{itemize}
%Thus, if we assume the validity of the Leopoldt-Schneider Conjecture when $j >0$, we can define $\widehat G^\varepsilon_j$ simply by the set of characters $\chi \in \widehat G^\varepsilon$ satisfying
%$${\rm ord}_{s=j} L_{K,S,T}(\chi, s)=\begin{cases}
%r_j^\varepsilon &\text{ if $j \leq 0$,}\\
%0 &\text{ if $j>0$}.
%\end{cases}$$
\end{remark}
%For each integer $j$ we define
%
%\[ Y_{L,W}(j) := \begin{cases} \bigoplus_{w \in W_0}\ZZ_p[G] w, &\text{if $j = 0$,}\\
%                               Y_L(j), &\text{otherwise}\end{cases}\]
%

In the remainder of this section we define for each integer $j$ a canonical isomorphism of $\CC_p[G]$-modules
\[ \lambda_{j} : \varepsilon_j\CC_p {\bigwedge}_{\ZZ_p[G]}^{r^\varepsilon_j} H^1(\co_{L,S},\ZZ_p(1-j))
\stackrel{\sim}{\to} \varepsilon_j
\CC_p {\bigwedge}_{\ZZ_p[G]}^{r^\varepsilon_j} Y_{L}(-j).\]
%
%(The superscript `BK' in this notation stands for `Bloch-Kato' and is to indicate that the isomorphisms described below are those which occur in the
%relevant case of the Tamagawa number conjecture of Bloch and Kato.)
%
The explicit definition that we give is motivated by the Tamagawa number conjecture of Bloch and Kato (see, in particular,  the proof of Corollary \ref{etnc proof} below).

\subsubsection{The case $j<0$} In this case the known validity of the Quillen-Lichtenbaum Conjecture (which follows from the recent proof by
Rost and Voevodsky of the Bloch-Kato Conjecture) gives a canonical Chern character isomorphism
\[ {\rm ch}_j: K_{1-2j}(\co_{L})\otimes_\ZZ\ZZ_p \stackrel{\sim}{\to} H^1(\co_{L,S},\ZZ_p(1-j))\]
where we write $K_*(-)$ for Quillen's higher algebraic $K$-theory functor.

One also has $\varepsilon_j = \varepsilon$ (by Remark \ref{exp idem}(i)) and
we define $\lambda_{j}$ to be the $r^\varepsilon_j$-th exterior power of the composite isomorphism of $\CC_p[G]$-modules
\begin{equation*}\label{regulator1} \varepsilon\CC_p H^1(\co_{L,S},\ZZ_p(1-j)) \stackrel{\sim}{\to}  \varepsilon\CC_p  K_{1-2j}(\co_L) \stackrel{\sim}{\to} \varepsilon\CC_p  H_L(-j)^+ \stackrel{\sim}{\to}
\varepsilon\CC_p
 Y_{L}(-j),\end{equation*}
where the first map is induced by the inverse of the isomorphism ${\rm ch}_j^{-1}$, the second by $(-1)$-times the Borel regulator map
\[ b_{j}: \RR K_{1-2j}(\co_L) \stackrel{\sim}{\to} \RR H_L(-j)^+\]
and the third by the isomorphism in (\ref{beta}).

\subsubsection{The case $j = 0$} We note that $ H^1(\co_{L,S},\ZZ_p(1))$ is identified with $\ZZ_p \co_{L,S}^\times$ via Kummer theory and we define
$\lambda_{0}$ to be the $r^\varepsilon_0$-th exterior power of the composite isomorphism of $\CC_p[G]$-modules
\begin{equation}
\varepsilon_0 \CC_p  H^1(\co_{L,S},\ZZ_p(1)) =  \varepsilon_0 \CC_p \co_{L,S}^\times \stackrel{\sim}{\to} \varepsilon_0 \CC_pX_{L,S} \stackrel{\sim}{\to} \varepsilon_0\CC_p Y_{L}(0)
 \label{log}\end{equation}
where the first map is the restriction of the Dirichlet regulator (sending each $a$ in $\co_{L,S,T}^\times$ to $- \sum_{w \in S_L} \log |a|_w w$)
and the second isomorphism follows from the vanishing of $\varepsilon_0 \CC_p X_{L,S\setminus S_\infty}$ (see Remark \ref{exp idem}(ii)).

%We also set $\mathcal{A}:=e\ZZ_p[G]$ for simplicity.

%We now assume that $j\neq 0$ and define a canonical isomorphism
%\begin{eqnarray}
%\CC_p \bigwedge_{\mathcal{A}}^r eH^1_T(\co_{L,S},\ZZ_p(1-j)) \stackrel{\sim}{\to} \CC_p \bigwedge_{\mathcal{A}}^r eY_{L}(j) \label{regulator}
%\end{eqnarray}
%as follows.

\subsubsection{The case $j = 1$} \label{subsection j=1}
%For simplicity we assume ${\bf 1}_G \notin\hat G_1^\varepsilon$ (which is the case, by Remark \ref{exp idem}(iii), if $T = \emptyset$) and only discuss the general case in Remark \ref{general case} below.
We write $\Gamma_{L,S}$ for the Galois group of the maximal abelian pro-$p$ extension of $L$ unramified outside $S$. Then the module $H^1(\co_{L,S},\QQ_p)$ identifies with $\Hom_{\rm cont}(\Gamma_{L,S},\QQ_p)$ and so, by combining the
%\textcolor{red}{reciprocity theorem of}
%\textcolor{green}{(words with red colour can (should) be removed)}
global class field theory with Remark \ref{exp idem}(iii), one obtains a canonical short exact sequence of $\CC_p[G]$-modules
\begin{equation}\label{j=1 iso}
\xymatrix{ \varepsilon_1 \CC_p H^1(\co_{L,S},\QQ_p) \ar@{^{(}->}[r]  & \varepsilon_1{\bigoplus}_{w\in S_p(L)} (\CC_p \mathcal{O}_{L_w}^\times)^\ast  \ar[d]^{\simeq}_{{\rm exp}^\ast_p} \ar@{->>}[r] & \varepsilon_1(\CC_p\mathcal{O}_L^\times)^\ast \\
 & \varepsilon_1(\CC_p\otimes_{\QQ}L)^\ast  & \varepsilon_1(\CC_p H_L(0)^+)^\ast .\ar[u]_{\simeq}
 }
\end{equation}
Here the first vertical isomorphism is induced by the linear dual of the $p$-adic exponential map homomorphisms $L_w\to \QQ_p\mathcal{O}_{L_w}^\times$ for $w$ in $S_p(L)$ and the second by the linear dual of the isomorphism $\CC_p\mathcal{O}_L^\times \simeq \CC_p X_{L,S_\infty}$ induced by the Dirichlet regulator map and the fact that
% that our assumption ${\bf 1}_G \notin \hat G_1^\varepsilon$ implies
$\varepsilon_1 \QQ_p X_{L,S_\infty}$ is equal to $\varepsilon_1\QQ_p Y_{L}(0)$ and hence isomorphic to $\varepsilon_1 \QQ_p H_L(0)^+$ by (\ref{beta}).

Abbreviating ${\rm det}_{\CC_p[G]}(-)$ to ${\rm D}(-)$, we then define $\lambda_{1}$ to be the composite isomorphism of $\CC_p[G]$-modules
\begin{multline*} \varepsilon_1 \CC_p{\bigwedge}_{\ZZ_p[G]}^{r^\varepsilon_1} H^1(\co_{L,S},\ZZ_p) = \varepsilon_1{\rm D}(\CC_p H^1(\co_{L,S},\ZZ_p))\\ \simeq \varepsilon_1\left({\rm D}((\CC_p\otimes_{\QQ}L)^\ast)\otimes_{\CC_p[G]}{\rm D}^{-1}((\CC_p H_L(0)^+)^\ast)\right)\\
 \simeq \varepsilon_1 {\rm D}((\CC_p H_L(1)^+)^\ast) \simeq \varepsilon_1{\rm D}(\CC_pY_L(1)^\ast) \simeq \varepsilon_1 \CC_p{\bigwedge}_{\ZZ_p[G]}^{r^\varepsilon_1}Y_L(-1).\end{multline*}
Here the first isomorphism is the canonical isomorphism induced by (\ref{j=1 iso}), the second is induced by the linear dual of (\ref{period map}) (with $j =1$), the third by (\ref{beta}) and the last by the canonical identification $Y_L(1)^\ast \simeq Y_L(-1)$.

\subsubsection{The case $j>1$} %Note that we have the exact triangle
%$$R\Gamma_f(L, \QQ_p(1-j)) \to R\Gamma(\co_{L,S},\QQ_p(1-j)) \to \bigoplus_{w\in S_L}R\Gamma_{/f}(L_w, \QQ_p(1-j)).$$
%(See \cite[(26)]{BFetnc} for example.)
In this case the vanishing of $\varepsilon_j H^2(\co_{L,S},\QQ_p(1-j))$ (see Remark \ref{exp idem}) combines with the
 local and global duality theorems to give a canonical short exact sequence of $\CC_p[G]$-modules

{\small
\begin{equation*}
\xymatrix{ \varepsilon_j \CC_p H^1(\co_{L,S},\ZZ_p(1-j)) \ar@{^{(}->}[r]  &  \ar[d]^{\simeq}_{{\rm syn}_{p}}\varepsilon_j \bigoplus_{w\in S_p(L)} \CC_pH^1(L_w,\ZZ_p(j))^\ast \ar@{->>}[r]
& \varepsilon_j \CC_p H^1(\co_{L,S},\ZZ_p(j))^\ast \\
 & \varepsilon_j(\CC_p\otimes_{\QQ}L)^\ast  & \varepsilon_j\CC_p H_L(j-1)^{+,\ast} \ar[u]_{\simeq} }
\end{equation*}
}
in which the second vertical homomorphism is induced by the dual of $-b_{1-j} \circ {\rm ch}_{1-j}^{-1}$ and the
first is induced by the linear duals for each $w$ in $S_p(L)$ of the canonical composite homomorphisms $L_w \to H^1_{\rm syn}(\mathcal{O}_{L_w},j) \to H^1(L_w,\QQ_p(j))$ involving syntomic cohomology that are discussed by Besser in \cite[(5.3) and Cor. 9.10]{besser}.

We then define $\lambda_{j}$ to be the isomorphism of $\CC_p[G]$-modules
%
%\[ \varepsilon_j \CC_p{\bigwedge}_{\ZZ_p[G]}^{r^\varepsilon_j} H_T^1(\co_{L,S},\ZZ_p(1-j)) \simeq \varepsilon_j %\CC_p{\bigwedge}_{\ZZ_p[G]}^{r^\varepsilon_j}Y_L(-j)\]
%
obtained from the above diagram in just the same way that $\lambda_{1}$ is obtained from
(\ref{j=1 iso}).% in the case ${\bf 1}_G\notin \hat G_1^\varepsilon$.

\begin{remark}\label{besser rem} In \cite[Prop. 9.11]{besser} Besser proves that for $w$ in $S_p(L)$ the composite homomorphism $L_w \to H^1_{\rm syn}(\mathcal{O}_{L_w},j) \to H^1(L_w,\QQ_p(j))$ used above coincides with the exponential map of Bloch and Kato for $\QQ_p(j)$ over $L_w$. In this way the definition of $\lambda_{j}$ for $j > 1$ is naturally analogous to the definition of
$\lambda_{1}$. \end{remark}

\begin{remark} \label{remark schneider}
  A closer analysis of the discussions used to define $\lambda_{j}$ for $j>0$ shows that, in this case, if $\varepsilon Y_L(1-j)$ vanishes, then $\varepsilon_j=\varepsilon$. %Indeed, in this case we have
%$$\varepsilon \CC_p K_{2j-1}(\co_L)^\ast \simeq \varepsilon \CC_p H_L(j-1)^{+,\ast} \simeq \varepsilon \CC_p Y_L(j-1)^\ast \simeq \varepsilon \CC_p Y_L(1-j) =0,$$
%and so
%\begin{multline*}
%\varepsilon  H^1(\co_{L,S},\CC_p(1-j)) \simeq \varepsilon \bigoplus_{w \in S_p(L)}H_f^1(L_w, \CC_p(j))^\ast \simeq \varepsilon (\CC_p \otimes_\QQ %L)^\ast \\ \simeq \varepsilon \CC_p H_L(j)^{+,\ast} \simeq \varepsilon \CC_p Y_L(-j) \simeq (\CC_p[G]\varepsilon)^{\oplus r_j^\varepsilon}.
%\end{multline*}
\end{remark}

\subsection{The definition of generalized Stark elements}

%\begin{definition} Fix $w_0$ in $S_L$. Then the `Stark element of rank $r$ and weight $0$' for the data $(L/K,S,T,e)$
%is the pre-image $\epsilon_{L/K,S,T}^W(0)$ of $e\theta_{L/K,S,T}^{(r)}(0)\cdot (w_1-w_0)\wedge \cdots \wedge (w_r-w_0)$ under the isomorphism $\CC_p\bigwedge_{\ZZ_p[G]}^r H^1_T(\co_{L,S},\ZZ_p(1)) \stackrel{\sim}{\to} \CC_p \bigwedge_{\ZZ_p[G]}^r X_{L,S}$ induced by
%(\ref{log}). \end{definition}

%, and for each integer $r$, we set
%$$\theta_{L/K,S,T}^{(r)}(s):=\frac{1}{s^r}\theta_{L/K,S,T}(s).$$

\begin{definition}\label{gen stark def} Fix an integer $j$, an idempotent $\varepsilon$ of $\ZZ_p[G]$ that satisfies Hypothesis \ref{hypo} (with respect to $j$). Fix also finite sets $S$ and $T$ of places of $K$ satisfying
\begin{itemize}
\item $S_\infty(K)\cup S_p(K)\cup S_{\rm ram}(L/K) \subset S$;
\item $S\cap T=\emptyset$;
\item $T=\emptyset$ if $j=1$.
\end{itemize}
Then the `Stark element of rank $r^\varepsilon_j$ and weight $-2j$' for $(L/K,S,T,\varepsilon)$ is the unique element $\eta_{L/K,S,T}^{\varepsilon}(j)$ of
$\varepsilon_j\CC_p {\bigwedge}_{\ZZ_p[G]}^{r^\varepsilon_j} H^1(\co_{L,S},\ZZ_p(1-j))$ that satisfies
 \[ \lambda_{j}(\eta_{L/K,S,T}^{\varepsilon}(j)) = \varepsilon_j\theta_{L/K,S,T}^{\ast}(j)\cdot
 {\bigwedge}_{w \in W_j^{\varepsilon}}w(-j) \ \ \text{ in }\ \ \varepsilon_j\CC_p{\bigwedge}_{\ZZ_p[G]}^{r^\varepsilon_j}Y_L(-j). \]%\in \CC_p \bigwedge_\mathcal{A}^r eY_{L}(j).\]
\end{definition}

\begin{remark}
It is natural to regard $\eta_{L/K,S,T}^\varepsilon(j)$ to be of weight $-2j$ since it is associated to the motive $h^0(\Spec L)(j)$.
\end{remark}

\begin{example}\label{example2} Definition \ref{gen stark def} generalizes the classical notion of Rubin-Stark element introduced by Rubin in \cite{rubinstark}. In fact, we have
$$\varepsilon_0\theta_{L/K,S,T}^{\ast}(0) =
\varepsilon \cdot \lim_{s \to 0 }s^{-r_0^\varepsilon}\theta_{L/K,S,T}(s)$$
by Remark \ref{exp idem}(ii) and \cite[Chap. I, Prop. 3.4]{tatebook} and so $\eta_{L/K,S,T}^{\varepsilon}(0)$ coincides with (the `$\varepsilon$-component' of) the Rubin-Stark element for the data
 $(L/K,S,T,W_0^\varepsilon)$.
 %In addition, with $\varepsilon = e^-$ one has in this case $r_0^\varepsilon=0$ and
 %$\eta_{L/K,S,T}^{\varepsilon}(0)= \varepsilon_0 \theta^\ast_{L/K,S,T}(0)=\theta_{L/K,S,T}(0).$
\end{example}

The following proposition is a natural analogue of \cite[Prop. 6.1]{rubinstark}.

\begin{proposition} \label{property stark}
Suppose that  $(L'/K,S',T',\varepsilon')$ is another collection of data as in Definition \ref{gen stark def} (with respect to $j$) for which all of the following properties are satisfied:  $L \subset L'$, $S\subset S'$, $T \subset T'$, with $G':=\Gal(L'/K)$ the natural surjection $\ZZ_p[G'] \to \ZZ_p[G]$ sends $\varepsilon'$ to $\varepsilon$ and $W_j^{\varepsilon}$ is the set of places of $L$ obtained by restricting places in $W_j^{\varepsilon'}$.

Then $r_j^{\varepsilon'}=r_j^\varepsilon=:r$ and the map
$$\varepsilon_j' \CC_p {\bigwedge}_{\ZZ_p[G']}^{r} H^1(\co_{L',S'},\ZZ_p(1-j)) \to \varepsilon_j\CC_p {\bigwedge}_{\ZZ_p[G]}^{r} H^1(\co_{L,S'},\ZZ_p(1-j))$$
induced by the corestriction map ${\rm Cor}_{L'/L}: H^1(\co_{L',S'}, \ZZ_p(1-j)) \to H^1(\co_{L,S'},\ZZ_p(1-j))$ sends $\eta_{L'/K,S',T'}^{\varepsilon'}(j)$  to
$$ \delta_{L/K,T'\setminus T}(j) \cdot \left( \prod_{v \in S' \setminus S}(1-{\N}v^{-j}{\rm Fr}_v^{-1}) \right) \cdot \eta_{L/K,S,T}^\varepsilon(j).$$
\end{proposition}

\begin{proof} This follows easily from the fact that the natural surjection $\CC[G'] \to \CC[G]$ sends $\varepsilon_j' \theta_{L'/K,S',T'}^\ast(j)$ to $$\varepsilon_j \theta_{L/K,S',T'}^\ast(j)=\varepsilon_j \delta_{L/K, T'\setminus T}(j) \cdot \left( \prod_{v \in S' \setminus S}(1-{\N}v^{-j}{\rm Fr}_v^{-1}) \right) \cdot \theta_{L/K,S,T}^\ast(j).$$
\end{proof}

\section{Statement of the conjectures}\label{statement of conjs}

\subsection{A Rubin-Stark Conjecture in arbitrary weight}

\subsubsection{Exterior power biduals and pairings} Fix a commutative ring $R$ and a finitely generated $R$-module $M$. In the following, we abbreviate $\bigwedge_R^r (M^\ast)$ to $\bigwedge_R^r M^\ast$.
%In the following, ${\bigwedge}_R^r M^\ast$ is understood to be the $r$-th exterior power of $M^\ast=\Hom_R(M,R)$.

For non-negative integers $r$ and $s$ with $r\le s$ there is a canonical pairing
$${\bigwedge}_R^s M \times {\bigwedge}_R^r M^\ast \to {\bigwedge}_R^{s-r} M$$
defined by
$$(a_1\wedge\cdots\wedge a_s,\varphi_1\wedge\cdots\wedge \varphi_r)\mapsto
\sum_{\sigma \in {\mathfrak{S}_{s,r}} }{\rm{sgn}}(\sigma)\det(\varphi_i(a_{\sigma(j)}))_{1\leq i,j\leq r} a_{\sigma(r+1)}\wedge\cdots\wedge a_{\sigma(s)} ,$$
with $\mathfrak{S}_{s,r}:= \{ \sigma \in \mathfrak{S}_s  \mid \sigma(1) < \cdots < \sigma(r) \text{ and } \sigma(r+1) <\cdots<\sigma(s) \}.$ We denote the image of $(a,\Phi)$ under the above pairing by $\Phi(a)$.

We also use the following construction (compare \cite[\S 1.2]{rubinstark}).

\begin{definition} For each non-negative integer $r$ the $r$-th exterior bidual of $M$ is the module
$${\bigcap}_R^r M:=\left({\bigwedge}_R^r M^\ast \right)^\ast.$$
%For $a$ in ${\bigcap}_R^r M$ and $\Phi$ in ${\bigwedge}_R^r M^\ast$, we set $\Phi(a) := a(\Phi) \in R$.
\end{definition}

%\begin{definition}
%Define the $r$-th Rubin lattice of $M$ by
%$$\bigcap_\mathcal{A}^r M:=\{ a \in \QQ_p{\bigwedge}_\mathcal{A}^r M \mid \Phi(a)\in \mathcal{A} \text{ for every $\Phi \in {\bigwedge}_\mathcal{A}^r (M)^\ast$}\}. $$
%\end{definition}

\begin{remark} If $R=\ZZ_p[G]$ with a finite abelian group $G$, then the map $a \mapsto (\Phi \mapsto \Phi(a))$ induces an identification
\[ \left\{ a \in \QQ_p{\bigwedge}_{\ZZ_p[G]}^r M \ \middle| \ \Phi(a)\in \ZZ_p[G] \text{ for every $\Phi \in {\bigwedge}_{\ZZ_p[G]}^r M^\ast$} \right\} \simeq {\bigcap}_{\ZZ_p[G]}^r M\]
and so we may regard ${\bigcap}_{\ZZ_p[G]}^r M$ as a subset of $\QQ_p{\bigwedge}_{\ZZ_p[G]}^r M$. \end{remark}

\begin{lemma} \label{lemma phi}
Suppose that $R=\ZZ_p[G]$ with a finite abelian group $G$ and that $M$ is $\ZZ_p$-free. Let $H$ be a subgroup of $G$ and denote the natural surjection $\QQ_p[G] \to \QQ_p[G/H]$ by $\pi_H$. Then, for any $a \in \QQ_p \bigwedge_{\ZZ_p[G]}^r M$, we have
$$\pi_H\left( \left\{ \Phi(a) \ \middle| \ \Phi \in {\bigwedge}_{\ZZ_p[G]}^r M^\ast \right\}\right)=\left\{ \Psi ({\N}_H^r (a)) \ \middle| \ \Psi \in {\bigwedge}_{\ZZ_p[G/H]}^r (M^H)^\ast \right\},$$
where
$${\N}_H^r : \QQ_p {\bigwedge}_{\ZZ_p[G]}^r M \to \QQ_p {\bigwedge}_{\ZZ_p[G/H]}^r M^H$$
is the map induced by the norm map
$$M \to M^H; \ m \mapsto \sum_{\sigma \in H} \sigma \cdot m.$$
\end{lemma}

\begin{proof}
This follows from \cite[Rem. 2.9 and Lem. 2.10]{sano}.
\end{proof}

\subsection{$T$-modified cohomology} \label{ST} Let $j$ be an integer, and $S$ and $T$ sets of places of $K$ as in Definition \ref{gen stark def}.
%In the sequel we fix a finite set $T$ of places of $K$ with the following properties:
%\begin{itemize}
%\item $S \cap T=\emptyset$;
%\item $S$ contains $S_\infty(K)\cup S_p(K) \cup S_{\rm ram}(L/K)$;
%\item $T=\emptyset$ if $j=1$.
%\end{itemize}

Let now $R$ denote any of the rings $\ZZ_p$, $\QQ_p$ and $\ZZ/p^n$ for some
natural number $n$. Then, as $T$ is disjoint from $S$,
for each $w$ in $T_L$
%the natural localization morphism $R\Gamma(\co_{L,S},R(a)) \to
%R\Gamma(L_w,R(a))$ factors through the natural `inflation' morphism
%$R\Gamma(\kappa(w),R(a)) \to R\Gamma(L_w,R(a))$ and so induces a morphism
there is a natural momorphism of \'{e}tale cohomology complexes
$R\Gamma(\co_{L,S},R(1-j)) \to R\Gamma(\kappa(w),R(1-j))$.

We define
 $R\Gamma_{T}(\co_{L,S},R(1-j))$ to be a complex that lies in an exact triangle in the derived category $D(R[G])$ of complexes of $R[G]$-modules of the form
\begin{equation}\label{T-mod def} R\Gamma_{T}(\co_{L,S},R(1-j)) \to R\Gamma(\co_{L,S},R(1-j)) \to
\bigoplus_{w\in T_L}R\Gamma(\kappa(w),R(1-j)) \to , \end{equation}
where the second arrow is the diagonal map induced by the morphisms described above. In each degree $i$ we then set
\[ H^i_T(\co_{L,S},R(1-j)) := H^i(R\Gamma_{T}(\co_{L,S},R(1-j)))\]
and we note that $H_T^i(\co_{L,S},\QQ_p(1-j))=H^i(\co_{L,S},\QQ_p(1-j))$ (this follows from the fact that $R\Gamma(\kappa(w),\QQ_p(1-j))$ is acyclic if $j\neq 1$ and the assumption that $T=\emptyset$ if $j=1$). In particular, we can regard $\bigcap_{\ZZ_p[G]}^r H^1_T(\co_{L,S},\ZZ_p(1-j))$ as a lattice of $\QQ_p \bigwedge_{\ZZ_p[G]}^r H^1(\co_{L,S},\ZZ_p(1-j))$.

\begin{example} \label{ex ST}
Kummer theory identifies
$H_T^1(\co_{L,S},\ZZ_p(1))$ with the $p$-completion of the $(S,T)$-unit group $\co_{L,S,T}^\times:=\ker(\co_{L,S}^\times \to \bigoplus_{w\in T_L}\kappa(w)^\times)$
of $L$.
\end{example}

%\begin{remark} \label{prop torsion}
%Then $eH_T^1(\co_{L,S},\ZZ_p(1-j))$ is $\ZZ_p$-torsion-free if and only if the map
%$$eH^0(L,\QQ_p/\ZZ_p(1-j)) \to eH^1(\co_{L,S},\ZZ_p(1-j)) \to e\bigoplus_{w\in T_L}H_f^1(L_w,\ZZ_p(1-j))$$
%is injective, where the first map is the boundary map, and the second is the localization.
%\end{remark}

%Now assume that $\varepsilon \in I_G \cap I_T$ if $a=0$,
%where $I_G:=\ker (\ZZ_p[G] \to \ZZ_p)$ and $I_T:=\bigcap_{v\in T} \ker (\ZZ_p[G] \to \ZZ_p[G/G_v])$ with $G_v \subset G$ the decomposition subgroup at a place of $L$ lying above $v$. (When $T=\emptyset$, we regard $I_T=\ZZ_p[G]$.) Then we have
%$$\varepsilon H^0(\co_{L,S}, \ZZ_p)=\varepsilon \bigoplus_{w\in T_L} H^0(\kappa(w), \ZZ_p)=0,$$
%and so for all $a$ the long exact cohomology sequence of the triangle (\ref{T-mod def}) gives an exact sequence
%
%\begin{multline}\label{lesc} 0 \to \varepsilon H^1_{T}(\co_{L,S},\ZZ_p(a)) \to \varepsilon H^1(\co_{L,S},\ZZ_p(a)) \to \varepsilon \bigoplus_{w\in T_L}H^1(\kappa(w),\ZZ_p(a)) \\ \to
%\varepsilon H^2_{T}(\co_{L,S},\ZZ_p(a)) \to \varepsilon H^2(\co_{L,S},\ZZ_p(a)) \to 0.\end{multline}
%Note that in this case $\varepsilon \bigoplus_{w\in T_L} H^1(\kappa(w),\ZZ_p(a))$ is finite, so we have
%$$\varepsilon \QQ_p H^1_T(\co_{L,S}, \ZZ_p(a))= \varepsilon \QQ_p H^1(\co_{L,S},\ZZ_p(a)).$$
%and we can regard $\varepsilon \bigcap_{\ZZ_p[G]}^r H^1_T(\co_{L,S},\ZZ_p(1-j))$ as a lattice of $\varepsilon \QQ_p {\bigwedge}_{\ZZ_p[G]}^r H^1(\co_{L,S},\ZZ_p(1-j))$.

\subsubsection{Statement of the conjecture}
 In the sequel for each non-negative integer $i$ we write ${\rm Fitt}_{\ZZ_p[G]}^i(M)$ for the $i$-th Fitting ideal of a $\ZZ_p[G]$-module $M$. We also write $I_G$ for the augmentation ideal of $\ZZ_p[G]$.

\begin{conjecture}\label{higherfitt} Fix an integer $j$, an idempotent $\varepsilon$ of $\ZZ_p[G]$ satisfying Hypothesis \ref{hypo} (with respect to $j$) and sets of places $S$ and $T$ as in Definition \ref{gen stark def}. Assume $\varepsilon H_T^1(\co_{L,S},\ZZ_p(1-j))$ is $\ZZ_p$-free and in addition that $\varepsilon \in I_G$ if $j = 1$.

Then, with $\eta = \eta_{L/K,S,T}^\varepsilon(j)$, one has
\begin{equation}\label{bks eq}
\left\{\!\Phi(\eta) \middle| \Phi\! \in\!
{\bigwedge}_{\ZZ_p[G]}^{r^\varepsilon_j}H^1_T(\co_{L,S},\ZZ_p(1-j))^\ast\!\right\}\!=\!\varepsilon\cdot{\rm Fitt}_{\ZZ_p[G]}^{0}(H^2_{T}(\co_{L,S},\ZZ_p(1-j)))
\end{equation}
and hence also
\begin{equation}\label{rs inclusion} \eta \in {\bigcap}_{\ZZ_p[G]}^{r^\varepsilon_j} H^1_T(\co_{L,S},\ZZ_p(1-j)).\end{equation}
\end{conjecture}

\begin{remark}\label{injectivity}
One sees that $H^1(\co_{L,S},\ZZ_p)$ is always $\ZZ_p$-free. When $j \neq 1$, we see that
$\varepsilon H_T^1(\co_{L,S},\ZZ_p(1-j))$ is $\ZZ_p$-free if and only if the
composite map
\[ \varepsilon H^0(L,\QQ_p/\ZZ_p(1-j)) \to
\varepsilon H^1(\co_{L,S},\ZZ_p(1-j)) \to \varepsilon \bigoplus_{w\in T_L}H^1(\kappa(w),\ZZ_p(1-j))\]
is injective, where the first map is the natural boundary homomorphism.
% and the last is induced by the triangle (\ref{T-mod def}).
 \end{remark}

 \begin{remark}
 %The case of $j=1$ is exceptional in the context of Conjecture \ref{higherfitt} since this is the only case in which $H^0(\co_{L,S},\ZZ_p(1-j))$ does not vanish. %(According to a part of the Tamagawa number conjecture \cite[Conjecture 4(ii)]{BFetnc}, this phenomenon is conceptually due to the existence of a simple pole of the corresponding $L$-function at $s=1$.)
% However, if $\varepsilon \in I_G$, then $\varepsilon H^0(\co_{L,S},\ZZ_p)$ vanishes
 %. (Correspondingly, the $L$-function $\varepsilon \theta_{L/K,S,T}(s)$ does not possess a pole at $s=1$.)
 %and this fact enables us
With the assumption that $\varepsilon \in I_G$ if $j=1$, we can use the complex $R\Gamma_T(\co_{L,S},\ZZ_p(1-j))[1] \oplus Y_L(-j)[-1]$ to construct an exact sequence of $\ZZ_p[G] \varepsilon$-modules of the form
 $$0 \to \varepsilon H^1_T(\co_{L,S},\ZZ_p(1-j)) \to F \to F \to \varepsilon (H_T^2(\co_{L,S},\ZZ_p(1-j))\oplus Y_L(-j) )\to 0$$
 where $F$ is both finitely generated and free (see \S \ref{zeta evidence}). This sequence is a natural analogue of classical `Tate sequences' (as discussed, for example, in \cite[\S 2.3]{bks1}) and plays a key role in our analysis. (For $j=1$ and the trivial character, a Tate sequence of similar kind is studied in \cite{GreitherK2} by Greither and the second author.)
\end{remark}

\begin{example} \label{example rubin stark}
If we identify $H_T^1(\co_{L,S},\ZZ_p(1))$ with the $p$-completion of the $(S,T)$-unit group $\co_{L,S,T}^\times$ of $L$ (see Example \ref{ex ST}) then in the case of $j =0$ the equality (\ref{bks eq}) recovers the `$\varepsilon$-component' of
the $p$-completion of \cite[Conj. 7.3]{bks1} and thus constitutes a refinement of a range of well-known conjectures
in the literature. For the same reason, in the setting of Example \ref{example2}, the $j=0$ case of the containment (\ref{rs inclusion}) recovers the
`$\varepsilon$-component' of the $p$-completion of the Rubin-Stark Conjecture \cite[Conj. B$'$]{rubinstark} for the data $(L/K,S,T,W_0^\varepsilon)$.
\end{example}

\begin{example} \label{minus example}Assume $K$ totally real, $L$ CM, $j\le 0$ and take $\varepsilon$ to be the idempotent $e_j^-$ in Example
\ref{example}(ii).

\noindent{}(i) In this case the inclusion (\ref{rs inclusion}) is unconditionally valid. To see this, note $r_j^\varepsilon = 0$ so
\[\eta_{L/K,S,T}^\varepsilon(j)=e^-_j\theta_{L/K,S,T}^\ast(j)=\theta_{L/K,S,T}(j)=\delta_{L/K,T}(j)\theta_{L/K,S}(j).\]
In addition, there is a natural exact sequence
\[ 0 \to \bigoplus_{v \in T} \ZZ_p[G] \xrightarrow{(1-{\N}v^{1-j}{\rm Fr}^{-1}_v)_v } \bigoplus_{v \in T}\ZZ_p[G] \to \bigoplus_{w\in T_L}H^1(\kappa(w),\ZZ_p(1-j)) \to 0\]
which implies
\begin{equation}\label{T equality} \delta_{L/K,T}(j)\cdot \ZZ_p[G]\varepsilon = \varepsilon {\rm Fitt}_{\ZZ_p[G]}^0\left(\bigoplus_{w\in T_L}H^1(\kappa(w),\ZZ_p(1-j)) \right)\end{equation}
and the assumed injectivity of the displayed map in Remark \ref{injectivity} implies this ideal is contained in $\varepsilon
{\rm Ann}_{\ZZ_p[G]}(H^0(L,\QQ_p/\ZZ_p(1-j)))$.
% since the assumption that $\varepsilon \cdot H_T^1(\co_{L,S},\ZZ_p(1-j))$ is $\ZZ_p$-free
The claimed inclusion (\ref{rs inclusion}) thus follows from the fact that Deligne and Ribet \cite{DR} have shown that
 $a \cdot \theta_{L/K,S}(j)$ belongs to $\ZZ_p[G]$ for any $a$ in ${\rm Ann}_{\ZZ_p[G]}(H^0(L,\QQ_p/\ZZ_p(1-j)))$. %{\color{red} precise reference please}

\noindent{}(ii) If $j < 0$, then the conjectural equality (\ref{bks eq}) implies
\begin{multline*} \ZZ_p[G]\cdot\theta_{L/K,S,T}(j) =\varepsilon{\rm Fitt}_{\ZZ_p[G]}^0(H_T^2(\co_{L,S},\ZZ_p(1-j))) \\
\subset \varepsilon{\rm Fitt}_{\ZZ_p[G]}^0(H^2(\co_{L,S},\ZZ_p(1-j))) = \varepsilon\ZZ_p{\rm Fitt}_{\ZZ[G]}^0(K_{-2j}(\co_{L,S})),\end{multline*}
where the inclusion is true as $H^2(\co_{L,S},\ZZ_p(1-j))$ is a quotient of $H^2_T(\co_{L,S},\ZZ_p(1-j))$ and the
final equality because the validity of the
Quillen-Lichtenbaum Conjecture gives a canonical isomorphism $H^2(\co_{L,S},\ZZ_p(1-j)) \simeq \ZZ_p K_{-2j}(\co_{L,S})$. Noting that $K_{-2j}(\co_{L}) \subset K_{-2j}(\co_{L,S})$ (see \cite[Prop. 5.7]{CKPS}), this displayed inclusion shows (\ref{bks eq}) refines the classical Coates-Sinnott Conjecture, which predicts $\ZZ_p[G]\cdot \theta_{L/K,S,T}(j) \subset \ZZ_p {\rm Ann}_{\ZZ[G]}(K_{-2j}(\co_L)).$
\end{example}

In \S\ref{zeta evidence} we interpret generalized Stark elements in terms of the theory of arithmetic zeta elements and use this connection to obtain the following evidence in support of Conjecture \ref{higherfitt}.

\begin{theorem}\label{evidence}
Conjecture 3.5 is valid in both of the following cases.
\begin{itemize}
\item[(i)] $L$ is an abelian extension of $\QQ$.
\item[(ii)] $K$ is totally real, $L$ is CM, $j\le 0$, $\varepsilon$ is
the idempotent $e_j^-$ in Example \ref{example}(ii), and
the Iwasawa $\mu$-invariant vanishes for the cyclotomic
$\ZZ_{p}$-extension $L_{\infty}/L$.
\end{itemize}
\end{theorem}

We end this section by stating some functorial properties of Conjecture \ref{higherfitt}.

\begin{proposition} \label{functorial1}
Let $(L'/K,S',T', \varepsilon')$ be as in Proposition \ref{property stark}.
\begin{itemize}
\item[(i)] Suppose $S'=S$ and $T'=T$. Then (\ref{bks eq}) (resp. (\ref{rs inclusion})) for $(L'/K,S,T,\varepsilon' ,j )$ implies (\ref{bks eq}) (resp. (\ref{rs inclusion})) for $(L/K,S,T,\varepsilon,j )$.
\item[(ii)] Suppose that $L'=L$, $T'=T$ and $\varepsilon'=\varepsilon$. Then (\ref{rs inclusion}) for $(L/K,S,T,\varepsilon,j)$ implies that for $(L/K,S',T,\varepsilon,j)$.
\item[(iii)] Suppose that $L'=L$, $S'=S$ and $\varepsilon'=\varepsilon$. Then (\ref{rs inclusion}) for $(L/K,S,T,\varepsilon,j)$ implies that for $(L/K,S,T',\varepsilon,j)$.
\end{itemize}
\end{proposition}

\begin{proof}
We know that $\varepsilon' R\Gamma_T(\co_{L',S},\ZZ_p(1-j))$ is a perfect complex of $\ZZ_p[G']$-modules and acyclic outside degrees one and two (see Lemma \ref{prop complex} below), and that
$$R\Gamma_T(\co_{L',S},\ZZ_p(1-j))\otimes_{\ZZ_p[G']}^{\mathbb{L}} \ZZ_p[G] \simeq R\Gamma_T(\co_{L,S},\ZZ_p(1-j))$$
(see \cite[Prop. 1.6.5]{FK}, for example). From this
we see that
$$\varepsilon' H^1_T(\co_{L',S},\ZZ_p(1-j))^{\Gal(L'/L)} \simeq \varepsilon H^1_T(\co_{L,S},\ZZ_p(1-j))$$
and that
$$\varepsilon' H^2_T(\co_{L',S},\ZZ_p(1-j)) \otimes_{\ZZ_p[G']} \ZZ_p[G] \simeq \varepsilon H^2_T(\co_{L,S},\ZZ_p(1-j)).$$
Noting this, claim (i) follows from Proposition \ref{property stark} and Lemma \ref{lemma phi}.

Next, we show claim (ii). We have an exact sequence
$$0 \to H_T^1(\co_{L,S},\ZZ_p(1-j)) \to H^1_T(\co_{L,S'},\ZZ_p(1-j)) \to \bigoplus_{w \in (S'\setminus S)_L} H^1_{/f}(L_w, \ZZ_p(1-j)).$$
Since the last term is $\ZZ_p$-free, we see that the restriction map
$$H^1_T(\co_{L,S'},\ZZ_p(1-j))^\ast \to H^1_T(\co_{L,S},\ZZ_p(1-j))^\ast$$
is surjective. From this, we see that
$${\bigcap}_{\ZZ_p[G]}^r H^1_T(\co_{L,S},\ZZ_p(1-j)) \subset {\bigcap}_{\ZZ_p[G]}^r H^1_T(\co_{L,S'},\ZZ_p(1-j)).$$
Now the assertion in claim (ii) is clear by Proposition \ref{property stark}.
%it is clear that (\ref{rs inclusion}) for $(L/K,S,T,\varepsilon)$ implies that for $(L/K,S',T,\varepsilon)$.

One can prove claim (iii) in the same way as \cite[Prop. 5.3.1]{popescuBC} by using the exact triangle (\ref{T-mod def}) and the equality (\ref{T equality}).
%
%For simplicity, we denote the left hand side of the equality (\ref{bks eq}) by $\Omega_{L/K,S,T}^\varepsilon(j)$.
%It is sufficient to prove that
%$$
%\Omega_{L/K,S',T'}^\varepsilon(j)= \delta_{L/K,T'\setminus T}(j) \cdot \left( \prod_{v\in S'\setminus S}(1-{\N}v^{-j}{\rm Fr}_v^{-1}) \right) \cdot \Omega_{L/K,S,T}^\varepsilon(j)
%$$
%and that
%\begin{multline*}
%\varepsilon{\rm Fitt}_{\ZZ_p[G]}^{0}(H^2_{T'}(\co_{L,S'},\ZZ_p(1-j))) \\
%=\delta_{L/K,T'\setminus T}(j) \cdot \left( \prod_{v\in S'\setminus S}(1-{\N}v^{-j}{\rm Fr}_v^{-1}) \right) \cdot \varepsilon{\rm Fitt}_{\ZZ_p[G]}^{0}(H^2_{T}(\co_{L,S},\ZZ_p(1-j))).
%\end{multline*}
%These equalities follow from the proof of Theorem \ref{thm1} and the exact triangles
%$$C_{L,S, T} (j) \to C_{L,S',T}(j) \to \bigoplus_{w \in (S'\setminus S)_L} R\Gamma_{/f}(L_w, \ZZ_p(1-j))[1] \to$$
%and
%$$C_{L,S', T'} (j) \to C_{L,S',T}(j) \to \bigoplus_{w\in (T'\setminus T)_L} R\Gamma(\kappa(w), \ZZ_p(1-j))[1] \to.$$
%
%
\end{proof}

\subsection{Congruences between Stark elements of differing weights}\ \

\

The `refined class number formula for $\mathbb{G}_m$', as independently conjectured by Mazur-Rubin \cite{MRGm} and the third author \cite{sano},
constitutes a family of congruence relations between Stark elements of weight zero and differing ranks. In this section we formulate precise
families of conjectural congruences between Stark elements of fixed rank and different weights.

At the outset we fix an integer $j$, an idempotent $\varepsilon$ of $\ZZ_p[G]$ satisfying Hypothesis \ref{hypo} (with respect to $j$) and sets of places $S$ and $T$ as in \S\ref{ST}.

We set $r := r^\varepsilon_j$ and $W:=W_j^\varepsilon$ . We also fix a labeling $W=\{w_1,\ldots,w_r\}$ and use this to define the wedge product ${\bigwedge}_{w \in W}$.
%defined in \S\ref{period-regulator} belongs to $\ZZ_p[G]$.

%\subsubsection{Preliminaries}
%We write $p^n$ for the cardinality of the $p$-power torsion subgroup of $L^\times$ and
We fix a positive integer $n$ such that $\mu_{p^n}\subset L^\times$ and use the cyclotomic character
\[ \chi_{\rm cyc}: G \to \Aut(\mu_{p^n})\simeq (\ZZ/p^n)^\times.\]

For each integer $a$ we also write ${\rm tw}_a$ for the ring automorphism of $\ZZ/p^n[G]$ that sends
 each element $\sigma$ of $G$ to $\chi_{\rm cyc}(\sigma)^a \sigma$. We then fix an integer $k$ and define $\delta$ to be the unique idempotent of $\ZZ_p[G]$ which projects to ${\rm tw}_{k-j}(\overline \varepsilon)$ in $\ZZ/p^n[G]$, where $\overline \varepsilon$ denotes the image in $\ZZ/p^n[G]$ of $\varepsilon$.

 Then, by an explicit computation, one checks that  $W^{\delta}_{k} = W$ (and hence $r_k^{\delta}=r$) and that $\delta Y_L(-k)$ is a free $\ZZ_p[G]\delta$-module of rank $r$ with basis $\{\delta \cdot w(-k) \mid w \in W\}$.

We next define a ${\rm tw}_{k-j}$-semilinear homomorphism
\begin{eqnarray*}
{\rm tw}_{j,k}: \varepsilon {\bigcap}_{\ZZ_p[G]}^r H^1_T(\co_{L,S},\ZZ_p(1-j)) \to \delta{\bigcap}_{\ZZ/p^n[G]}^r
H^1_T(\co_{L,S},\ZZ/p^n(1-k)) \label{twist lattice}
\end{eqnarray*}
that will play a key role in our conjectural congruences.

For simplicity, for each integer $a$ we set
$$H(\ZZ_p(a)):= H^1_T(\co_{L,S},\ZZ_p(a))\,\,\,\text{ and } \,\,\, H(\ZZ/p^n(a)):= H^1_T(\co_{L,S},\ZZ/p^n(a)).$$
Note that the natural map $H(\ZZ_p(1-j))\to H(\ZZ/p^n(1-j))$ induces a homomorphism
\begin{eqnarray}
{\bigcap}_{\ZZ_p[G]}^r H(\ZZ_p(1-j)) \to {\bigcap}_{\ZZ/p^n[G]}^r H(\ZZ/p^n(1-j)).  \label{mod p^n}
\end{eqnarray}

Recall that each $w \in S_\infty(L)$ determines the embedding $\iota_{w}: L \hookrightarrow \CC$ (see \S \ref{setup}) and for
each integer $i$ with $1\leq i \leq r$, we set
\begin{eqnarray}
\xi_i:=\iota_{w_i}^{-1}(e^{2\pi\sqrt{-1}/p^n}) \in H^0(L,\ZZ/p^n(1)). \label{def xi}
\end{eqnarray}
We write
\begin{eqnarray*}
c_i:  H(\ZZ/p^n(1-k))^\ast \to H(\ZZ/p^n(1-j))^\ast \label{def ci}
\end{eqnarray*}
for the map induced by cup product with $\xi_i^{\otimes (j-k)}$ and
$${\bigwedge}_{\ZZ/p^n[G]}^r  H(\ZZ/p^n(1-k))^\ast \to {\bigwedge}_{\ZZ/p^n[G]}^r H(\ZZ/p^n(1-j))^\ast$$
for the map sending each element $a_1\wedge\cdots \wedge a_r$ to $c_1(a_1)\wedge \cdots \wedge c_r(a_r).$
%This map is well-defined since each $c_i$ is ${\rm tw}_{-k}$-semilinear i.e.
%$$c_i(\sigma a)=\kappa(\sigma)^{-k}\sigma c_i(a). $$
Taking the $\ZZ/p^n$-dual of the last map we obtain a homomorphism
$${\bigcap}_{\ZZ/p^n[G]}^r H(\ZZ/p^n(1-j)) \to {\bigcap}_{\ZZ/p^n[G]}^r  H(\ZZ/p^n(1-k))$$
%(Note that both the $\mathcal{A}/p^n$-dual and $\mathcal{A}_k/p^n$-dual are identified with the $\ZZ/p^n$-dual.)
%
and we define ${\rm tw}_{j,k}$ to be the composite of this homomorphism with (\ref{mod p^n}).

%\subsubsection{Statement of the conjecture}
%We continue to write $\ZZ/p^n$ in place of $\ZZ_p/p^n$ and $r$ in place of
%$r^\varepsilon_j = r^{\varepsilon^k_j}_{j+k}$.
%In the following conjecture we set $W_k:=(e_k, (w_1,\ldots,w_r))$.
%Note that, by Proposition \ref{rubin2}(i), there is a natural map
%$$\bigcap_{\mathcal{A}_k}^r e_kH_T^1(\co_{L,S},\ZZ_p(1-j-k)) \to \bigcap_{\mathcal{A}_k/p^n}^r e_kH^1_T(\co_{L,S},\ZZ/p^n(1-j-k)). $$

\begin{conjecture}\label{congruence conjecture} Fix an integer $j$, an idempotent $\varepsilon$ of $\ZZ_p[G]$ satisfying Hypothesis \ref{hypo} (with respect to $j$) and sets of places $S$ and $T$ as in \S\ref{ST}. Assume that the integer $k$, and associated idempotent $\delta$ defined above, are such that
\begin{itemize}
%\item $\varep$ if either $j$ or $k$ is equal to $1$;
\item $\varepsilon H^1_T(\co_{L,S},\ZZ_p(1-j))$ and $\delta H^1_T(\co_{L,S},\ZZ_p(1-k))$ are both $\ZZ_p$-free;
\item $T=\emptyset$ if either $j=1$ or $k=1$;
\item $\varepsilon\in I_G$ if $j=1$;
\item $\delta\in I_G$ if $k=1$.
\end{itemize}

Then, if the containment (\ref{rs inclusion}) is valid for both pairs $(\varepsilon,j)$ and
$(\delta, k)$, one has
$${\rm tw}_{j,k}(\eta_{L/K,S,T}^{\varepsilon}(j)) = \eta_{L/K,S,T}^{\delta}(k)$$
in the finite module $\delta{\bigcap}_{\ZZ/p^n[G]}^r  H^1_T(\co_{L,S},\ZZ/p^n(1-k))$.
\end{conjecture}

We discuss evidence for this conjecture in \S\ref{cong evi proof} and, in particular, prove the following result. This result (and its proof)
shows that the conjecture incorporates a wide selection of results ranging from classical explicit reciprocity law due to Artin-Hasse and Iwasawa
to the classical congruences of Kummer. In this result we use the notation of Example \ref{example}.

\begin{theorem}\label{cong evidence} Assume $K$ is totally real and $L$ is CM.
\begin{itemize}
\item[(i)] If $K = \QQ$, then for all integers $j$ and $k$ Conjecture \ref{congruence conjecture} is valid with $\varepsilon = e^+_j$.

\item[(ii)] For all non-positive integers $j$ and $k$
Conjecture \ref{congruence conjecture} is valid with $\varepsilon = e_j^-$.

\item[(iii)] Conjecture \ref{congruence conjecture} for the data $T = \emptyset, j = 0$, $\varepsilon = e^+$ and
$k = 1$ is a refinement of the `Congruence Conjecture' \cite[{${\rm CC}(L/K,S,p,n-1)$}]{solomon} of Solomon.
\end{itemize}
\end{theorem}

\begin{remark}\label{cong evid rem} The proof of claim (i) relies both (if $j < 0$) on results of Beilinson and H\"uber-Wildeshaus on the cyclotomic elements of
Deligne-Soul\'e and (if $j > 0$) on Kato's generalized explicit reciprocity law, whilst  claim (ii) relies on results of Deligne and Ribet.
Claim (iii) is of interest both because Solomon's Conjecture (recalled in \S\ref{solomon section} below) is formulated as an explicit reciprocity
law for Rubin-Stark elements (extending that of Artin-Hasse and Iwasawa \cite{iwasawa}) and also because it allows us to interpret
the extensive numerical evidence in support of Solomon's conjecture in \cite{rs} as evidence for Conjecture \ref{congruence conjecture}.\end{remark}

\begin{remark} In a subsequent article we intend to explore connections between Conjecture \ref{congruence conjecture} and the very general
(conjectural) formalism discussed by Fukaya and Kato in \cite{FK}.\end{remark}

\begin{proposition} \label{functorial2}
Let $(L'/K,S',T',\varepsilon')$ be as in Proposition \ref{property stark}.
\begin{itemize}
\item[(i)] Suppose $S'=S$ and $T'=T$. Then Conjecture \ref{congruence conjecture} for $(L'/K,S,T,\varepsilon',j,k)$ implies that for $(L/K,S,T,\varepsilon,j,k)$.
\item[(ii)] Suppose that $L'=L$ and $\varepsilon'=\varepsilon$. Then Conjecture \ref{congruence conjecture} for $(L/K,S,T,\varepsilon,j,k)$ implies that for $(L/K,S',T',\varepsilon,j,k)$.
%\item[(iii)] Suppose that $L'=L$, $S'=S$ and $\varepsilon'=\varepsilon$. Then Conjecture \ref{congruence conjecture} for $(L/K,S,T,\varepsilon,j,k)$ implies that for $(L/K,S,T',\varepsilon,j,k)$.
\end{itemize}
\end{proposition}

\begin{proof}
Note that, since $\mu_{p^n} \subset L$, ${\rm tw}_{j,k}$ is a $\ZZ_p[\Gal(L'/L)]$-homomorphism. Note also that ${\N}_{L'/L}^r (\eta_{L'/K,S,T}^{\varepsilon'}(j))=\eta_{L/K,S,T}^\varepsilon(j)$ and ${\N}_{L'/L}^r (\eta_{L'/K,S,T}^{\delta'}(j))=\eta_{L/K,S,T}^\delta(j)$ by Proposition \ref{property stark}, where $\N_{L'/L}^r=\N_{\Gal(L'/L)}^r$. Assuming Conjecture \ref{congruence conjecture} for $(L'/K,S,T,\varepsilon',j,k)$, we have
\begin{multline*}
{\rm tw}_{j,k}(\eta_{L/K,S,T}^\varepsilon(j))={\rm tw}_{j,k}({\N}_{L'/L}^r (\eta_{L'/K,S,T}^{\varepsilon'}(j))) \\
= {\N}_{L'/L}^r ({\rm tw}_{j,k}(\eta_{L'/K,S,T}^{\varepsilon'}(j))) \equiv {\N}_{L'/L}^r(\eta_{L'/K,S,T}^{\delta'}(j)  ) =  \eta_{L/K,S,T}^\delta(j)  \text{ (mod $p^n$)}.
\end{multline*}
Hence we have proved claim (i).

Since we have
$${\rm tw}_{k-j}(\delta_{L/K,T'\setminus T}(j)) \equiv \delta_{L/K,T'\setminus T}(k) \text{ (mod $p^n$)}$$
and
$${\rm tw}_{k-j}\left(\prod_{v \in S'\setminus S}(1-{\N} v^{-j} {\rm Fr}_v^{-1})\right) \equiv \prod_{v \in S'\setminus S}(1-{\N} v^{-k} {\rm Fr}_v^{-1}) \text{ (mod $p^n$)},$$
claim (ii) follows from the fact that ${\rm tw}_{j,k}$ is ${\rm tw}_{k-j}$-semilinear, and Propositions \ref{property stark} and \ref{functorial1}.
\end{proof}

\begin{proposition} Suppose that $v \notin S\cup T$ splits completely in $L$, and assume (\ref{rs inclusion}) for both $(L/K,S, T,\varepsilon,j)$ and $(L/K,S, T, \delta,k)$. Then Conjecture \ref{congruence conjecture} is valid for $(L/K,S\cup\{v\},T,\varepsilon,j,k)$. %and for $(L/K,S,T \cup \{v\},\varepsilon,j,k)$.
\end{proposition}

\begin{proof} If $v$ is any such place, then $p^n $ divides $\#\kappa(v)^\times={\N}v-1$ (since $\mu_{p^n}\subset L$) and hence also both $(1-{\N}v^{-j})$ and $(1-{\N}v^{-k})$.
The stated assumptions and Proposition \ref{property stark} therefore imply that
\[ \eta_{L/K,S\cup\{v\},T}^\varepsilon(j)=(1-{\N}v^{-j})\eta_{L/K,S,T}^\varepsilon(j)\in p^n\cdot {\bigcap}_{\ZZ_p[G]}^r H^1_T(\co_{L,S},\ZZ_p(1-j))\]
and that
\[ \eta_{L/K,S\cup \{v\},T}^{\delta}(k)=(1-{\N}v^{-k})\eta_{L/K,S,T}^{\delta}(k) \in p^n\cdot
{\bigcap}_{\ZZ_p[G]}^r H^1_T(\co_{L,S},\ZZ_p(1-k))\]
and so both sides of the displayed equality in Conjecture \ref{congruence conjecture} for $(L/K,S\cup\{v\}, T,\varepsilon,j,k)$ vanish.
%Similarly, we can show that the both sides of the equality in Conjecture \ref{congruence conjecture} for $(L/K,S, T\cup \{v\},\varepsilon,j,k)$ vanish.
\end{proof}

\section{Zeta elements and the proof of Theorem \ref{evidence}}\label{zeta evidence}

In this subsection, we interpret generalized Stark elements in terms of the theory of arithmetic zeta elements and use this connection to prove Theorem \ref{evidence}.

\subsection{Perfect complexes}

Let $\varepsilon \in \ZZ_p[G]$ be any idempotent (we do not need to assume Hypothesis \ref{hypo} in this subsection). With $Z$ denoting either $\ZZ_p$ or $\ZZ/p^n$ for some natural number $n$ we define an object of $D(Z[G]\varepsilon)$ by setting
\[ C_{L,S,T}^{\varepsilon}(j)_Z := Z[G]\varepsilon \otimes^{\mathbb{L}}_{\ZZ_p[G]}( R\Gamma_{T}(\co_{L,S},\ZZ_p(1-j))[1] \oplus Y_L(-j)[-1] ) .\]
%We assume that $T=\emptyset$ if $a=1$.
%
%For convenience we also often abbreviate $C_{L,S,T}^{\varepsilon}(a)_Z$ to $C^{\varepsilon}(a)$.

The properties of these complexes that we use are recorded in the following result.% (and are essentially well-known).

We write $D^{\rm perf}(Z[G]\varepsilon)$ for the full triangulated subcategory of
$D(Z[G]\varepsilon)$ comprising complexes that are `perfect' (that is, isomorphic to a bounded complex of finitely generated projective $Z[G]\varepsilon$-modules).

\begin{lemma} \label{prop complex} The following claims are valid for all integers $j$.

\begin{itemize}
\item[(i)] $C_{L,S,T}^{\varepsilon}(j)_Z$ belongs to $D^{\rm perf}(Z[G]\varepsilon)$ and is acyclic outside degrees $-1, 0$ and $1$.
\item[(ii)] Assume that $\varepsilon H^1_T(\co_{L,S},\ZZ_p(1-j))$ is $\ZZ_p$-free if $Z=\ZZ/p^n$ and that $\varepsilon\in I_G$ if $j=1$. Then we have
$$H^{i}(C_{L,S,T}^\varepsilon(j)_Z)=\begin{cases}
0 &\text{if $i=-1$},\\
\varepsilon H^1_T(\co_{L,S},Z(1-j)) &\text{if $i=0$},\\
 \varepsilon H^2_T(\co_{L,S},Z(1-j)) \oplus \varepsilon (Y_L(-j)\otimes_{\ZZ_p}Z) &\text{if $i=1$}.
\end{cases}$$
Furthermore, we have a (non-canonical) isomorphism of $\QQ_p[G]$-modules
$$\QQ_p H^0(C_{L,S,T}^{\varepsilon}(j)_{\ZZ_p}) \simeq \QQ_p H^1(C_{L,S,T}^{\varepsilon}(j)_{\ZZ_p}).$$
\end{itemize}
\end{lemma}

\begin{proof} Since $p$ is odd, it is well-known that $R\Gamma(\co_{L,S},\ZZ_p(1-j))$ belongs to $D^{\rm perf}(\ZZ_p[G])$ and is acyclic outside
degrees zero, one and two (see, for example, \cite[Prop. 1.6.5]{FK}). Claim (i) follows from this and the
fact that the complex $\bigoplus_{w\in T_L}R\Gamma(\kappa(w),Z(1-j))$ in the triangle (\ref{T-mod def}) belongs to
$D^{\rm perf}(Z[G])$ and is acyclic outside degrees zero and one.

To prove the first assertion of claim (ii) it suffices to show $\varepsilon H^0_T(\co_{L,S}, Z(1-j))$ vanishes under the stated assumptions.

If $j\neq 1$, then $H^0(\co_{L,S}, \ZZ_p(1-j))$, and hence also, $H_T^0(\co_{L,S}, \ZZ_p(1-j))$ vanishes. If $\varepsilon H_T^1(\co_{L,S}, \ZZ_p(1-j))$ is $\ZZ_p$-free, then the exact triangle
$$R\Gamma_T(\co_{L,S},\ZZ_p(1-j)) \stackrel{p^n}{\to} R\Gamma_T(\co_{L,S},\ZZ_p(1-j)) \to R\Gamma_T(\co_{L,S},\ZZ/p^n(1-j)) \to$$
implies $\varepsilon H_T^0(\co_{L,S}, \ZZ/p^n(1-j))$ vanishes.

Next, we consider the case when $j=1$. Recall that we set $T=\emptyset$ in this case (see \S \ref{ST}). Since $\varepsilon \in I_G$ by assumption, we have
$$\varepsilon H^0_T(\co_{L,S},Z)=\varepsilon H^0(\co_{L,S},Z)=\varepsilon \cdot Z=0.$$
%If $T \neq \emptyset$, then by the triangle (\ref{T-mod def}) we have
%$$H^0_T(\co_{L,S}, Z)=\ker\left(Z \to \bigoplus_{w \in T_L} Z \right)=0.$$

%If $Z=\ZZ_p$, the assertion is obvious by the assumption that $\varepsilon \in I_G\cap I_T$ if $a=1$. So it is sufficient to prove that $\varepsilon H^0_T(\co_{L,S}, \ZZ/p^n(1-a))=0$. Since $\varepsilon H^1_T(\co_{L,S},\ZZ_p(1-a))$ is $\ZZ_p$-free by assumption, this follows from the exact triangle
%$$R\Gamma_T(\co_{L,S},\ZZ_p(1-a)) \stackrel{p^n}{\to} R\Gamma_T(\co_{L,S},\ZZ_p(1-a)) \to R\Gamma_T(\co_{L,S},\ZZ/p^n(1-a))\to 0. $$
%
%
%we see by Remark \ref{injectivity} that the natural map
%$$\varepsilon H^0(L, \QQ_p/\ZZ_p(1-a)) \to \varepsilon \bigoplus_{w \in T_L} H^0(\kappa(w),\QQ_p/\ZZ_p(1-a))$$
%is injective. From this, we see that
%$$\varepsilon H^0(L, \ZZ/p^n(1-a)) \to \varepsilon \bigoplus_{w \in T_L} H^0(\kappa(w),\ZZ/p^n(1-a))$$
%is also injective, and the exact triangle (\ref{T-mod def}) implies that $\varepsilon H^0_T(\co_{L,S}, \ZZ/p^n(1-a))=0$.
%The first assertion of claim (ii) is true because $H^{-1}(C_{L,S,T}^{\varepsilon}(a)_Z) = \varepsilon H^0_T(\co_{L,S},Z(1-a))$ is a submodule of
%$H^0(\co_{L,S},Z(1-a))$ and the latter module obviously vanishes if $a\not= 1$ and is annihilated by any element of $I_G$ if $a=1$.
%
To prove the remaining assertion of claim (ii) we write $R\Gamma_c(\co_{L,S},\ZZ_p(j))$ for the compactly supported cohomology complex of $\ZZ_p(j)$ and note that Artin-Verdier duality (as expressed, for example, in \cite[(6)]{BFetnc2}) combines with the triangle (\ref{T-mod def}) to give a
canonical exact triangle in $D^{\rm perf}(\ZZ_p[G])$ of the form
\begin{equation*}\bigoplus_{w\in T_L}R\Gamma(\kappa(w),\ZZ_p(1-j)) \to C_{L,S,T}^{1}(j)_{\ZZ_p} \to R\Hom_{\ZZ_p}(R\Gamma_c(\co_{L,S},\ZZ_p(j)),
\ZZ_p)[-2] \to . \end{equation*}
Then, since $C_{L,S,T}^{\varepsilon}(j)_{\ZZ_p}$ is acyclic outside degrees zero and one, the final assertion of claim (ii) follows from this triangle and the fact that the $\QQ_p[G]$-equivariant Euler characteristics of both $\varepsilon \bigoplus_{w\in T_L}R\Gamma(\kappa(w),\QQ_p(1-j))$ and $R\Gamma_c(\co_{L,S},\QQ_p(j))$ vanish. \end{proof}

%\subsubsection{L-functions}
%We set
%$$R\Gamma_{c,T}(\co_{L,S},\ZZ_p(j)):=R\Hom_{\ZZ_p}(C_{L,S,T}(j),\ZZ_p)[-2],$$
%and denote
%$$H^i_{c,T}(\co_{L,S},\ZZ_p(j)):=H^i(R\Gamma_{c,T}(\co_{L,S},\ZZ_p(j))).$$

\subsection{Zeta elements} We quickly review the definition of zeta elements in the context of Conjecture \ref{higherfitt}. To do this we fix
notation $L/K,G,p,S,T, j, \varepsilon$ and $\varepsilon_j$ as in \S \ref{section2}. We often abbreviate $C_{L,S,T}^\varepsilon(j)_{\ZZ_p}$ to $C_{L,S,T}^{\varepsilon}(j)$. When $\varepsilon=1$, we omit it from notations (so we denote $C_{L,S,T}^1(j)$ by $C_{L,S,T}(j)$, for example). % and abbreviate the complexes
%$C^{\ZZ_p}_{L,S,T,\varepsilon}(a)$ defined above to $C_\varepsilon(a)$.

The definition of $\varepsilon_j$ combines with Lemma \ref{prop complex}(ii) to imply $\QQ_p[G] \varepsilon_j \otimes_{\ZZ_p[G]}^\mathbb{L} C_{L,S,T}(j) $
is acyclic outside degrees zero and one and that there are canonical isomorphisms
\[ \varepsilon_j\QQ_p H^i(C_{L,S,T}(j))) \simeq \begin{cases}
\varepsilon_j  H_T^1(\co_{L,S},\QQ_p(1-j)) &\text{if $i = 0$,}\\
\varepsilon_j\QQ_p Y_{L}(-j) &\text{if $i = 1$.}\end{cases}\]
Since these $\QQ_p[G]\varepsilon_j$-modules are both free of rank $r^\varepsilon_j$ there is a canonical `passage to cohomology' isomorphism
 of $\QQ_p[G]\varepsilon_j$-modules
%Using these isomorphisms we obtain a canonical isomorphism of $\CC_p[G]$-modules
\begin{multline}\label{passage}
 \pi_j: \varepsilon_j \QQ_p {\det}_{\ZZ_p[G]}(C_{L,S,T}(j)) \\
  \xrightarrow{\sim} \varepsilon_j \QQ_p \left( {\bigwedge}_{\ZZ_p[G]}^{r^\varepsilon_j} H_T^1(\co_{L,S},\ZZ_p(1-j)) \otimes_{\ZZ_p[G]}
 {\bigwedge}_{\ZZ_p[G]}^{r^\varepsilon_j}  Y_{L}(j) \right) .\end{multline}
Here we identify $Y_{L}(j)$ with $Y_L(-j)^\ast$.

\begin{definition} \label{definition zeta element}
The zeta element associated to the data $(L/K,S,T,\varepsilon,j)$ is the unique element $z_{L/K,S,T}^{\varepsilon}(j)$ of
$\varepsilon_j\CC_p  {\det}_{\ZZ_p[G]}(C_{L,S,T}(j))$ that satisfies

$$\pi_j(z_{L/K,S,T}^\varepsilon(j))= \eta_{L/K,S,T}^\varepsilon(j) \otimes {\bigwedge}_{w \in W_j^{\varepsilon}} w(j),$$
or equivalently,
\[ ({\rm ev}_L \circ (\lambda_{j} \otimes {\rm id}) \circ \pi_j)(z_{L/K,S,T}^\varepsilon(j)) =
\varepsilon_j \theta_{L/K,S,T}^\ast(j), \]
where ${\rm ev}_{L}$ denotes the standard `evaluation' isomorphism
\[  {\bigwedge}_{\CC_p[G]}^{r^\varepsilon_j}\CC_pY_{L}(-j) \otimes_{\CC_p[G]}
 {\bigwedge}_{\CC_p[G]}^{r^\varepsilon_j} \CC_p Y_{L}(-j)^\ast
\simeq \CC_p[G]. \]
\end{definition}

%\begin{remark}
%One can define the canonical isomorphism
%$$\CC_p {\det}_{\ZZ_p[G]}(C_{L,S,T}(j)) \stackrel{\sim}{\to} \CC_p[G],$$
%and formulate the eTNC in a similar way without any assumptions. One can also show that the validity of the eTNC does not depend on the choices of %$S$ and $T$. See \cite{BFetnc} for the detail.
%\end{remark}

\subsection{The proof of Theorem \ref{evidence}} In this section we prove the following results.

\begin{theorem} \label{thm1} If there exists a $\ZZ_p[G] \varepsilon$-basis $z$ of $\varepsilon {\det}_{\ZZ_p[G]}(C_{L,S,T}(j))$ with
$\varepsilon_j z = z_{L/K,S,T}^\varepsilon(j)$, then Conjecture \ref{higherfitt} is valid.
\end{theorem}

\begin{corollary}\label{etnc proof} Theorem \ref{evidence} is valid.
\end{corollary}

\begin{proof} The first point to note is that the maps that are used in the explicit definition of the isomorphism $\lambda_{j}$ given in \S\ref{period-regulator} coincide with the maps that occur in the statement of the equivariant Tamagawa number conjecture for the pair $(h^0(\Spec L)(j),\ZZ_p[G]\varepsilon)$ (see \cite[Conj. 4(iv)]{BFetnc}). This fact is clear if $j \le 1$ and follows in the case $j > 1$ from the result of Besser recalled in Remark \ref{besser rem}.

Given this, and our definition of the element $z_{L/K,S,T}^\varepsilon(j)$, the latter conjecture implies the existence of
a $\ZZ_p[G]\varepsilon$-basis of $\varepsilon\cdot{\det}_{\ZZ_p[G]}(C_{L,S,T}(j))$ with the property stated in Theorem \ref{thm1}. We note that this conjecture is usually formulated without using the set $T$, but as noted in \cite[Prop. 3.4]{bks1} one can formulate a natural $T$-modified version of this conjecture,  whose validity is independent of the choice of $T$.

The result of Theorem \ref{evidence}(i) now follows directly from Theorem \ref{thm1} and the fact that if $L$ is abelian over $\QQ$, then the equivariant Tamagawa number conjecture for
$(h^0(\Spec L)(j),\ZZ_p[G])$ is known to be true (by work of the first author and Greither \cite{BG}, and of Flach \cite{flach2part}).

Theorem \ref{evidence}(ii) can be proved by the same method as in \cite[Cor. 3.18]{bks2}
by using the Iwasawa main conjecture proved by Wiles.
\end{proof}

The proof of Theorem \ref{thm1} occupies the rest of this section (and is motivated by the argument used to prove \cite[Th. 7.5]{bks1}). We assume that $\varepsilon H^1_T(\co_{L,S},\ZZ_p(1-j))$ is $\ZZ_p$-free and that $\varepsilon \in I_G$ if $j=1$.

We set $\mathcal{A} := \ZZ_p[G]\varepsilon$, $A := \QQ_p[G]\varepsilon$, $W:=W^\varepsilon_j$ and $r := r^\varepsilon_j$. We also label (and thereby order) the elements of $W$ as $\{w_i\}_{1\le i\le r}$.

Then Lemma \ref{prop complex}(ii) implies that $C_{L,S,T}^\varepsilon(j)$ is acyclic outside degrees zero and one and %that there is a natural surjective homomorphism of $\mathcal{A}$-modules
%$f_j: H^1(C_{L,S,T}^\varepsilon(j)) \to \varepsilon Y_L(-j)$.
we can therefore choose a representative of $C_{L,S,T}^\varepsilon(j)$ of the form $F \stackrel{\psi}{\to} F$
with $F$ a free $\mathcal{A}$-module with basis $\{b_1,\ldots,b_d\}$ for some sufficiently large integer $d$ %(so that we identify
%$H^1(C_\varepsilon(j))$ with $\coker(\psi)$)
so that
%for each integer $i$ with $1 \leq i \leq r$ the natural surjection
%
the natural surjection
\[ F \to \coker (\psi) = H^1(C_{L,S,T}^{\varepsilon}(j))= \varepsilon H^2_T(\co_{L,S},\ZZ_p(1-j)) \oplus \varepsilon Y_L(-j) \]
sends $b_i$ with $1\leq i \leq r$ to $\varepsilon \cdot w_i(-j)$ and $\{b_{r+1},\ldots ,b_d\}$ to a set of generators of $\varepsilon H^2_T(\co_{L,S},\ZZ_p(1-j))$.
%$$\langle b_{r+1},\ldots, b_d \rangle_\mathcal{A}  \to H^1(C_{L,S,T}^\varepsilon(j)) \stackrel{f_j}{\to} \varepsilon Y_L(-j) \to 0$$
%is exact, where the first map is induced by the natural map
%\[ F \to \coker (\psi) = H^1(C_{L,S,T}^{\varepsilon}(j)) ,\]
%
%and so that for each integer $i$ with $1 \leq i \leq r$ this map sends $b_i$ to a lift of $\varepsilon\cdot w_i(-j)$.
See \cite[\S 5.4]{bks1} for the detail of this construction. Note that the representative chosen in loc. cit. is of the form $P \to F$ with $P$ projective and $F$ free, but in the present case we can identify $P$ with $F$ by Swan's theorem (see \cite[(32.1)]{CR}). Also, note that the assumption that $\varepsilon H^1_T(\co_{L,S},\ZZ_p(1-j))$ is $\ZZ_p$-free is needed here.

We may therefore identify ${\det}_\mathcal{A}(C_{L,S,T}^{\varepsilon}(j))$ with ${\bigwedge}_\mathcal{A}^d F \otimes_\mathcal{A}
{\bigwedge}_\mathcal{A}^d F^\ast$. With respect to this identification, any
$\mathcal{A}$-basis of ${\det}_\mathcal{A}(C_{L,S,T}^{\varepsilon}(j))$ has the form
\[ z_x := x \cdot b_1\wedge\cdots\wedge b_d \otimes b_1^\ast \wedge \cdots \wedge b_d^\ast \]
with $x \in \mathcal{A}^\times$, where we write $b_i^\ast$ for the $\mathcal{A}$-linear dual of $b_i$.

Next we write
\[ \pi_j': \QQ_p {\rm det}_\mathcal{A}( C_{L,S,T}^{\varepsilon}(j)) \to \varepsilon_j \QQ_p {\rm det}_\mathcal{A}( C_{L,S,T}^{\varepsilon}(j)) \stackrel{\sim}{\to}
 \varepsilon_j{\bigwedge}_{\QQ_p[G]}^{r}H_T^1(\co_{L,S},\QQ_p(1-j))\]
for the composite homomorphism of $A$-modules in which the first map is `multiplication by $\varepsilon_j$' and the second is the composite of the isomorphism $\pi_j$ in (\ref{passage}) and
the isomorphism of $A$-modules $\varepsilon_j \QQ_p{\bigwedge}_{\ZZ_p[G]}^{r} Y_{L}(j) \stackrel{\sim}{\to} A\varepsilon_j$ that sends the element
$\varepsilon_j \cdot w_1(j) \wedge \cdots \wedge w_r(j)$ to $\varepsilon_j$.

Then, with this notation, the argument of \cite[Lem. 4.3]{bks1} implies that

\begin{eqnarray}\label{first step}\pi_j'( z_x) &=&(-1)^{r(d-r)}x \left( {\bigwedge}_{r< i \leq d}\psi_i \right)(b_1\wedge\cdots \wedge b_d) \\
&=& (-1)^{r(d-r)} x \sum_{\sigma \in \mathfrak{S}_{d,r}}
{\rm sgn}(\sigma)\det(\psi_i(b_{\sigma(k)}))_{r < i,k\leq d} b_{\sigma(1)}\wedge \cdots \wedge b_{\sigma(r)}\nonumber
\end{eqnarray}
%(See Notation.)
%Here the equality holds in
%$$(\QQ_p {\bigwedge}_{\mathcal{A}}^r eH^1(\co_{L,S},\ZZ_p(1-j))) \cap ({\bigwedge}_{\mathcal{A}}^r F)=\bigcap_{\mathcal{A}}^r %eH^1_T(\co_{L,S},\ZZ_p(1-j)).$$
%(See \cite[Lemma 4.7(ii)]{bks1}.) This shows that Conjecture ${\rm I}(L/K,S,T,j,W)$ holds.
with $\psi_i:=b_i^\ast \circ \psi \in F^\ast$ for each index $i$.

In particular, note that the element $( {\bigwedge}_{r< i \leq d}\psi_i )(b_1\wedge\cdots \wedge b_d)  $ of ${\bigwedge}_\mathcal{A}^r F$ lies in $\varepsilon_j{\bigwedge}_{\QQ_p[G]}^{r}H_T^1(\co_{L,S},\QQ_p(1-j))$, which is regarded as a submodule of $\QQ_p \bigwedge_\mathcal{A}^r F$ via the inclusion
\begin{eqnarray} \label{psi inclusion}
\varepsilon H_T^1(\co_{L,S},\ZZ_p(1-j))=H^0(C_{L,S,T}^\varepsilon(j)) =\ker \psi \hookrightarrow F.
\end{eqnarray}

Next we note that the matrix of the endomorphism $\psi$ with respect to the basis $\{b_1,\ldots,b_d\}$ of $F$ is $(\psi_i(b_k))_{1\leq i,k\leq d}$ and
 that $\psi_i=0$ for each $i$ with $1\leq i \leq r$ since the elements $\{ \varepsilon \cdot w_i(-j)\}_{1\le i\le r}$ are an $\mathcal{A}$-basis of $\varepsilon Y_{L}(-j)$. The
 matrices $\{ \det(\psi_i(b_{\sigma(k)}))_{r<i,k\leq d}\}_{\sigma \in \mathfrak{S}_{d,r}}$ are therefore a set of generators of the $\mathcal{A}$-module
\[ {\rm Fitt}_\mathcal{A}^r(H^1(C_{L,S,T}^{\varepsilon}(j))) = {\rm Fitt}_\mathcal{A}^0(\varepsilon H_T^2(\mathcal{O}_{L,S},\ZZ_p(1-j))) = \varepsilon  {\rm Fitt}_{\ZZ_p[G]}^0(H^2_T(\mathcal{O}_{L,S},\ZZ_p(1-j))),\]
where the first equality is valid since the $\mathcal{A}$-module $H^1(C_{L,S,T}^{\varepsilon}(j))$ is the direct sum of $\varepsilon H_T^2(\mathcal{O}_{L,S},\ZZ_p(1-j))$ and a free module $\varepsilon Y_L(-j)$ of rank $r$.

Note that the restriction map $F^\ast \to \varepsilon H_T^1(\co_{L,S},\ZZ_p(1-j))^\ast$ is surjective since the cokernel of (\ref{psi inclusion}) is $\ZZ_p$-free. This fact combines with the equality (\ref{first step})
%and the argument used to prove \cite[Theorem 7.5]{bks1}
to imply that
\begin{multline}\label{first version}
\left\{\Phi(\pi_j'(z_x)) \ \middle| \ \Phi\! \in\! \varepsilon {\bigwedge}_{\ZZ_p[G]}^r H^1_T(\mathcal{O}_{L,S},\ZZ_p(1-j))^\ast \right\} \\=
\varepsilon {\rm Fitt}_{\ZZ_p[G]}^0(H^2_T(\mathcal{O}_{L,S},\ZZ_p(1-j))).
\end{multline}

Now suppose that $\varepsilon_j \cdot z_x= z_{L/K,S,T}^\varepsilon(j)$. Then the definition of $z_{L/K,S,T}^\varepsilon(j)$ implies
$\pi_j'(z_x) =\eta_{L/K,S,T}^\varepsilon(j)$ and, given this,
%To complete the proof of Theorem \ref{thm1} we note
%
%\[ \lambda^{\rm BK}(j)\circ \pi_\varepsilon(j) =
%({\rm ev}_{L,j}\circ (\lambda^{\rm BK}(j)\otimes {\rm id}) \circ (\CC_p\otimes_{\QQ_p}\iota_{L,S,j,\varepsilon}))
%\]
%
%and hence that
%
%\[ \lambda^{\rm BK}(j)(\pi_\varepsilon(j)(\varepsilon_j\cdot z_x)) =
%({\rm ev}_{L,j}\circ (\lambda^{\rm BK}(j)\otimes {\rm id}) \circ (\CC_p\otimes_{\QQ_p}\iota_{L,S,j,\varepsilon}))(\varepsilon_j\cdot z_x). \]
%
%In particular, for any value of $x$ for which one has
%$\varepsilon_j\cdot z_x = z_{L/K,S,T,\varepsilon}(j)$ the definition of
 % $z_{L/K,S,T,\varepsilon}(j)$ implies $\lambda^{\rm BK}(j)(\pi_\varepsilon(j)(\varepsilon_j\cdot z_x)) = \vartheta_j\cdot \theta_{L/K,S,T}^*(j)$
 % and hence, by the injectivity of $\lambda^{\rm BK}(j)$, that $\pi_\varepsilon(j)(\varepsilon_j\cdot z_x) = \eta_{L/K,S,T}^\varepsilon(j)$.
 the result of Theorem \ref{thm1} follows directly from the equality (\ref{first version}).

\section{Some evidence for Conjecture \ref{congruence conjecture}}\label{cong evi proof}

In this section, we give a proof of Theorem \ref{cong evidence}. Throughout this section we assume $K$ totally real and $L$ CM.

\subsection{Deligne-Soul\'e elements, explicit reciprocity and Theorem \ref{cong evidence}(i)} In this subsection we assume $K=\QQ$ and $\varepsilon=e_j^+$. In this case we have $r:=r_j^{\varepsilon}=\# S_\infty(L)/G=\# S_\infty(\QQ)=1$ (see Example \ref{example}(i)).

By Proposition \ref{functorial2}, we may assume $L =\QQ(\mu_f)$ with $f \in \ZZ_{>0}$ such that $f \not\equiv 2 \ (\text{mod } 4)$. Also, we may assume $p^n \mid f$ and that $S$ is equal to the minimal set $\{\infty\}\cup \{\ell \mid f\}$, with $\infty$ the archimedean place of $\QQ$. Finally, we may assume $T=\emptyset$ since $p$ is odd. We note that $\varepsilon \in I_G$ is satisfied when $j=1$, and that $\varepsilon_j=\varepsilon$ holds when $j\neq 0$ (see Remarks \ref{exp idem}(i) and \ref{remark schneider}). In the following, we often omit $T$. (For example, we denote $\eta_{L/\QQ,S,\emptyset}^\varepsilon(j)$ by $\eta_{L/\QQ,S}^\varepsilon(j)$.) For simplicity, we denote $\eta_{L/\QQ,S}^{e_j^+}(j)$ by $\eta_{L/\QQ,S}^+(j)$.

Recall that $w \in S_\infty(L)/G$ determines the embedding $\iota_w : \overline \QQ \hookrightarrow \CC $ (see \S \ref{setup}). We set $\zeta_m:=\iota_w^{-1}(e^{2\pi \sqrt{-1}/m})$ for any integer $m$.

Recall the definition of `cyclotomic elements' of Deligne-Soul\'e.
%$$c_{1-j}(\zeta_f) \in H^1(\co_{L,S},\ZZ_p(1-j)).$$
First, for a positive integer $m$, define
$$c_{1-j}(\zeta_f)_m:={\rm Cor}_{\QQ(\mu_{p^m f})/L}((1-\zeta_{p^mf} )\otimes \zeta_{p^m}^{\otimes (-j)})\in H^1(\co_{L,S},\ZZ/p^m(1-j)). $$
Here we regard $(1-\zeta_{p^m f} )\otimes \zeta_{p^m}^{\otimes (-j)}$ as an element of $H^1(\ZZ[\mu_{p^m f},1/p],\ZZ/p^m (1-j))$ via the Kummer map
\begin{eqnarray*}
 \ZZ \left[\mu_{p^m f}, \frac1p \right]^\times \otimes \ZZ/p^m (-j)
&\to& H^1\left(\ZZ\left[\mu_{p^m f},\frac1p \right],\ZZ/p^m (1) \right)\otimes_\ZZ \ZZ/p^m (-j)\\
&\simeq& H^1\left(\ZZ\left[\mu_{p^m f},\frac1p\right],\ZZ/p^m (1-j)\right).
\end{eqnarray*}
The cyclotomic element is defined by the inverse limit
$$c_{1-j}(\zeta_f):=\varprojlim_m c_{1-j}(\zeta_f)_m \in \varprojlim_m H^1(\co_{L,S},\ZZ/p^m(1-j)) \simeq H^1(\co_{L,S},\ZZ_p(1-j)). $$

Noting that
$$\eta_{L/\QQ,S}^{+}(0)=2^{-1} \cdot (1-\zeta_f)(1-\zeta_f^{-1}) \in \ZZ_p \co_{L^+,S}^\times=e^+ H^1(\co_{L,S},\ZZ_p(1))$$
(see Example \ref{example2} and \cite[p.79]{tatebook}),
we see by definition that ${\rm tw}_{0,j}(\eta_{L/\QQ,S}^{+}(0))=e_j^+ c_{1-j}(\zeta_f)_n$. From this, we have
$${\rm tw}_{j,k}(e_j^+ c_{1-j}(\zeta_f))=e_k^+ c_{1-k}(\zeta_f)_n$$
for arbitrary integers $j$ and $k$.

Hence, it is sufficient to show that
\begin{eqnarray}
e_j^+c_{1-j}(\zeta_f)= \eta_{L/\QQ,S}^{+}(j). \nonumber
\end{eqnarray}

We may assume $j \neq 0$. Suppose first that $j <0$. In this case, by the definition of $\eta_{L/\QQ,S}^{+}(j)$, it is sufficient to show that the image of $e_j^+ c_{1-j}(\zeta_f)$ under the isomorphism
$$\lambda_{j}: e_j^+\CC_p H^1(\co_{L,S},\ZZ_p(1-j)) \simeq e_j^+\CC_p K_{1-2j}(\co_L) \simeq e_j^+\CC_p Y_L(-j),$$
%where the second isomorphism is the minus-times Borel regulator,
is $e_j^+\theta_{L/\QQ,S}^\ast(j) \cdot w(-j)$. This is a direct consequence of the results of Beilinson and H\"uber-Wildeshaus \cite[Cor. 9.7]{HW} (see also \cite[Th. 5.2.1 and 5.2.2]{HK}).

Next, suppose $j >0$. Again, it is sufficient to show that the image of $e_j^+ c_{1-j}(\zeta_f)$ under the isomorphism $\lambda_{j}$ is $e_j^+ \theta_{L/\QQ,S}^\ast(j) \cdot w(-j)$. Recalling Remark \ref{besser rem}, we note that in this case $\lambda_{j}$ coincides with the map
$$e_j^+\CC_p H^1 (\co_{L,S},\ZZ_p(1-j)) \stackrel{{\rm exp}_p^\ast}{\simeq}  e_j^+( \CC_p \otimes_\QQ L)^\ast \stackrel{\alpha_j^\ast}{\simeq} e_j^+ \CC_p H_L(j)^{+,\ast}\simeq e_j^+\CC_p Y_L(-j),$$
where ${\rm exp}_p^\ast$ is the dual exponential map, $\alpha_j^\ast$ is induced by  (\ref{period map}) %(restricted on $e_j^+$-component),
 and the last isomorphism is induced by (\ref{beta}) with the identification $Y_L(j)^\ast=Y_L(-j)$.

By using the explicit reciprocity law due to Kato \cite[Th. 5.12]{katoiwasawa} (see also \cite[Chap. II, Th. 2.1.7]{katolecture} and \cite[Th. 3.2.6]{HK}), we have
$${\rm exp}_p^\ast (c_{1-j}(\zeta_f))=\left(x \mapsto -\frac{1}{f^j} {\rm Tr}_{L/\QQ}(x d_j(\zeta_f))\right) \in L^\ast,$$
where $d_j$ is the polylogarithmic  function
$$d_j(t):=\frac{(-1)^j}{(j-1)!}  {\rm Li}_{1-j}(t)=\frac{(-1)^j}{(j-1)!} \left(\frac1t \frac{d}{dt}\right)^{j-1}\left(\frac{t}{1-t}\right). $$
From this and the classical formula
$$e_j^+\theta_{L/\QQ,S}^\ast(j)=\frac14 \left(\frac{2\pi \sqrt{-1}}{f}\right)^j \sum_{1\leq a \leq f, \ (a,f)=1}(d_j(e^{2\pi \sqrt{-1}a/f})+(-1)^j d_j (e^{-2\pi \sqrt{-1}a/f})) \sigma_a^{-1},$$
where $\sigma_a \in \Gal(L/\QQ)$ is the automorphism sending $\zeta_f$ to $\zeta_f^a$, we see by computation that
$\alpha_j^\ast\circ {\rm exp}_p^\ast(e_j^+ c_{1-j}(\zeta_f))=e_j^+ \theta_{L/\QQ,S}^\ast(j) \cdot w(-j)$ and this completes the proof.

\subsection{Generalized Kummer congruences and Theorem \ref{cong evidence}(ii)} In the setting of Theorem \ref{cong evidence}(ii) one has $\eta_{L/K,S,T}^{e_j^-}(j)=\theta_{L/K,S,T}(j)$ for any non-positive integer $j$ (see Example \ref{minus example}) and $r_j^{e_j^-}=0$ and the map ${\rm tw}_{j,k}$ coincides with the composite homomorphism
$$e_j^-\ZZ_p[G] \to e_j^- \ZZ/p^n[G] \stackrel{{\rm tw}_{k-j}}{\to} e_k^- \ZZ/p^n[G]$$
which sends each $\sigma \in G$ to $\chi_{\rm cyc}(\sigma)^{k-j}\sigma$. Hence, it is sufficient to show that
$${\rm tw}_j(\theta_{L/K,S,T}(0)) \equiv \theta_{L/K,S,T}(j) \text{ (mod $p^n$)}$$
for any non-positive integer $j$.

But, since ${\rm tw}_j(\delta_{L/K,T}(0)) \equiv \delta_{L/K,T}(j) \text{ (mod $p^n$)},$ the above congruence follows directly from the well-known result due to Deligne-Ribet \cite{DR} that for any element $a$ of ${\rm Ann}_{\ZZ_p[G]}(H^0(L,\QQ_p/\ZZ_p(1)))$ one has ${\rm tw}_j(a \cdot \theta_{L/K,S}(0)) \equiv {\rm tw}_j(a) \theta_{L/K,S}(j) \text{ (mod $p^n$)}.$
%So we have
%$$\epsilon_{E_\infty/K,S,T}^W(0)=\varprojlim_n \theta_{L_n/K,S,T}(0)=:\theta_{E_\infty/K,S,T}(0) \in \Lambda$$
%and
%$$c_{L/K,S,T}^W(0)_k={\rm tw}_k(\theta_{E_\infty/K,S,T}(0)) \in \ZZ_p[G].$$
%It is sufficient to prove that ${\rm tw}_k(\theta_{E_\infty/K,S,T}(0))=e_k\theta_{L/K,S,T}^\ast(k),$ where $e_k:=(1-(-1)^kc)/2$, and this is a  well-known property of the $p$-adic $L$-function of Deligne-Ribet.

\begin{remark} The above argument shows that Conjecture \ref{congruence conjecture} constitutes a wide-ranging extension of the classical Kummer congruences. To see this take $K=\QQ$, $L=\QQ(\mu_{p^n})$, $S=\{\infty,p\}$ and $T=\emptyset$. Then write $\Delta$ for the subgroup of $G$ of order $p-1$ and set
%Let $\omega \in \widehat \Delta$ be the Teichm\"uller character.
 $e_{\Delta}:=\frac{1}{p-1}\sum_{\sigma \in \Delta} \sigma \in \ZZ_p[G].$

Let $j,k$ be odd negative integers such that $j \equiv k$ (mod $p^{n-1}(p-1)$) and $1-j\not\equiv 0$ (mod $p-1$). The first condition implies that ${\rm tw}_{j-k}$ is the identity map, and the second that $e_{\Delta}H^0(L,\QQ_p/\ZZ_p(1-j))$ vanishes and hence that $e_\Delta H^1(\co_{L,S},\ZZ_p(1-j))$ is $\ZZ_p$-free. The same holds for $e_\Delta H^1(\co_{L,S},\ZZ_p(1-k))$ since $j \equiv k \text{ (mod $p-1$)}$.

By Theorem \ref{cong evidence}(ii), we deduce $e_{\Delta}\theta^\ast_{L/\QQ,S}(j) \equiv e_\Delta \theta^\ast_{L/\QQ,S}(k) \text{ (mod $p^n$)}$ and hence $(1-p^{-j})\zeta(j) \equiv (1-p^{-k})\zeta(k) \text{ (mod $p^n$)},$ where $\zeta(s)$ denotes the Riemann zeta function. This is exactly the formulation of Kummer's congruence.
\end{remark}

\subsection{Solomon's Congruence Conjecture and Theorem \ref{cong evidence}(iii)}\label{solomon section} In this subsection we first review the explicit statement of Solomon's Conjecture and then prove Theorem \ref{cong evidence}(iii).

%\begin{remark}
%Combining Theorems \ref{etnc implies twisting} and \ref{solomon thm}, we see that the eTNC implies Solomon's conjectures. Note that it has already been known that Solomon's Integrality Conjecture follows from the eTNC. This is due to Jones \cite{jones}. Our result newly gives that the `twisting eTNC' implies Solomon's Congruence Conjecture.
%\end{remark}
\subsubsection{}To review Solomon's conjecture we set $L_p:=L\otimes_\ZZ \ZZ_p \simeq \prod_{w \in S_p(L)}L_w$ and $U_{L_p}^1:=(\co_{L}\otimes_\ZZ \ZZ_p)^\times \otimes_\ZZ \ZZ_p \simeq \prod_{w \in S_p(L)}U_{L_w}^1$, where $U_{L_w}^1$ denotes the group of principal units of $L_w$.
%Let $M_S(L)/L$ be the maximal abelian pro-$p$ extension unramified outside $S$.
As in \S \ref{subsection j=1}, we denote by $\Gamma_{L,S}$ the Galois group of the maximal abelian pro-$p$ extension of $L$ unramified outside $S$.
%For any $\ZZ_p[G]$-module $M$, we often denote $e^{\pm} M$ by $M^{\pm}$, where $e^{\pm}:=(1\pm c)/2$.
We denote $\Gal(L^+/K)$ by $G^+$. We write $S_\infty(L)/G=\{w_1,\ldots,w_r\}$. Note that $r=[K:\QQ]$. %(see Example \ref{example}).

Note that $L/K$ corresponds to $K/k$ in \cite{solomon}. In \cite[\S 2.2]{solomon}, a representative $\{\tau_1,\ldots,\tau_r\}$ of the coset space $G_\QQ/G_K$ is fixed. To fit this choice into our setting, we set
$$\tau_i:=\iota_{w_1}^{-1}\circ \iota_{w_i}, $$
where each $\iota_{w_i}$ is regarded as an embedding $\overline \QQ \hookrightarrow \overline L_{w_i}=\CC$ (see \S \ref{setup}).
In \cite{solomon}, the algebraic closure of $\QQ$ is considered to be in $\CC$. This means that an embedding $\overline \QQ \hookrightarrow \CC$ is fixed. We take $\iota_{w_1}$ for this fixed choice.
Also, an embedding $j: \overline \QQ \hookrightarrow \overline \QQ_p$ is chosen in \cite[\S 2.4]{solomon}. Since we fixed an isomorphism $\CC \simeq \CC_p$, we take $j$ to be the composition $\overline \QQ \stackrel{\iota_{w_1}}{\hookrightarrow} \CC \simeq \CC_p.$

In \cite[Def. 2.14]{solomon}, Solomon defined a map
$$\mathfrak{s}_{L/K,S} \in \Hom_{\ZZ_p[G]^-}\left({\bigwedge}_{\ZZ_p[G]^-}^r U_{L_p}^{1,-},\QQ_p[G]^-\right)$$
by using the zeta value $e^-\theta_{L/K,S}^\ast(1)$ (we will give the definition in the proof of Proposition \ref{solomon stark} below). Solomon's Integrality Conjecture \cite[{${\rm IC}(L/K,S,p)$}]{solomon} is equivalent to the containment
\begin{eqnarray}
\mathfrak{s}_{L/K,S} \in \Hom_{\ZZ_p[G]^-}\left({\bigwedge}_{\ZZ_p[G]^-}^r U_{L_p}^{1,-}, \ZZ_p[G]^-\right). \nonumber
\end{eqnarray}

Next, we explain the formulation of Solomon's Congruence Conjecture. Assume that $\mu_{p^n} \subset L$. In \cite[\S 2.3]{solomon}, Solomon constructed a pairing
$$H_{L/K,S,n-1}: {\bigcap}_{\ZZ_p[G^+]}^r \ZZ_p\co_{L^+,S}^\times \times {\bigwedge}_{\ZZ_p[G]^-}^r U_{L_p}^{1,-} \to \ZZ/p^n[G]^-$$
by using the Hilbert symbol $L_w^\times \times L_w^\times \to \mu_{p^n}$ for $p$-adic places $w\in S_p(L)$ (see the proof of Proposition \ref{solomon stark} below).

Assuming the validity of (\ref{rs inclusion}) for the data $(L/K,S,\emptyset,e^+,0)$ (or equivalently, the $p$-part of the Rubin-Stark Conjecture for the data $(L^+/K,S,\emptyset, S_\infty(K))$, see Example \ref{example rubin stark}), Solomon's Congruence Conjecture \cite[{${\rm CC}(L/K,S,p,n-1)$}]{solomon} asserts that for every $u \in {\bigwedge}_{\ZZ_p[G]^-}^r U_{L_p}^{1,-}$ one has
$$\mathfrak{s}_{L/K,S}(u) \equiv (-1)^r \chi_{\rm cyc}(\tau_1\cdots \tau_r)H_{L/K,S,n-1}(\eta_{L/K,S}^{e^+}(0),u) \text{ (mod $p^n$)}.$$
(The sign $(-1)^r$ appears here since we use $(-1)$-times the usual logarithm
(see (\ref{log}))). %in the definition of the Rubin-Stark element $\epsilon_{L^+/K,S}^{S_\infty}=\epsilon_{L/K,S}^W(0)$, and hence we have $(-1)^r\epsilon_{L^+/K,S}^{S_\infty}=\eta_{L^+/K,S}$, where the right-hand-side is the Rubin-Stark element described in \cite[\S 2.2]{solomon}.

\subsubsection{}We now prove Theorem \ref{cong evidence}(iii). To do this we note that $H^1(\co_{L,S},\ZZ_p)^- = \Hom_{\rm cont}(\Gamma_{L,S}^-,\ZZ_p)$. The dual of the map ${\rm rec}_p: U_{L_p}^1 \to \Gamma_{L,S}$ that sends $u$ to $\prod_{w \in S_p(L)} {\rm rec}_w(u),$
where ${\rm rec}_w$ denotes the local reciprocity map at $w$, therefore induces a homomorphism
\begin{multline*}
{\rm rec}^\ast_p : {\bigcap}_{\ZZ_p[G]^-}^r H^1(\co_{L,S},\ZZ_p)^- \to {\bigcap}_{\ZZ_p[G]^-}^r \Hom_{\ZZ_p}(U_{L_p}^{1,-},\ZZ_p)\\
\simeq \Hom_{\ZZ_p[G]^-} \left({\bigwedge}_{\ZZ_p[G]^-}^r U_{L_p}^{1,-},\ZZ_p[G]^-\right)^\#,
\end{multline*}
where for a $\ZZ_p[G]$-module $M$ we denote by $M^\#$ the module $M$ on which $G$ acts via the involution $\sigma \mapsto \sigma^{-1}$, and the last isomorphism follows from
$$\Hom_{\ZZ_p}(U_{L_p}^{1,-},\ZZ_p) \simeq \Hom_{\ZZ_p[G]^-}(U_{L_p}^{1,-},\ZZ_p[G]^-)^\#; \ f \mapsto \sum_{\sigma\in G}f(\sigma(\cdot))\sigma^{-1}$$
and the definition of $r$-th exterior bidual. The proof of Theorem \ref{cong evidence}(iii) is thus reduced to the following result.

\begin{proposition} \label{solomon stark}One has ${\rm rec}_p^\ast (\eta_{L/K,S}^{e^-}(1))=(-1)^r\mathfrak{s}_{L/K,S}$ and
\begin{multline*}
{\rm rec}_p^\ast ( {\rm tw}_{0,1}(a))= \chi_{\rm cyc}(\tau_1\cdots \tau_r) H_{L/K,S,n-1}(a, \cdot)\\ \text{ in }\Hom_{\ZZ_p[G]^-}\left({\bigwedge}_{\ZZ_p[G]^-}^r U _{L_p}^{1,-},\ZZ/p^n[G]^- \right)^\#
\end{multline*}
for every $a\in {\bigcap}_{\ZZ_p[G]^+}^r H^1(\co_{L,S},\ZZ_p(1))^+={\bigcap}_{\ZZ_p[G^+]}^r \ZZ_p \co_{L^+,S}^\times$.
\end{proposition}

\begin{proof}
We review the definition of $\mathfrak{s}_{L/K,S}$. For an integer $i$ with $1\leq i \leq r$, we define
$$\log_p^{(i)}: U_{L_p}^{1,-} \to \overline \QQ_p$$
by $\log_p^{(i)}(u):=\log_p(j(\tau_i u))=\log_p(\iota_{w_i}(u)),$ where $\log_p$ denotes the $p$-adic logarithm defined on $\{a\in  \overline \QQ_p \mid |a-1|_p<1\}$. We define
$${\rm Log}_p^{(i)}:=\sum_{\sigma \in G}\log_p^{(i)} (\sigma (\cdot ))\sigma^{-1} \in \Hom_{\ZZ_p[G]^-}(U_{L_p}^{1,-},\overline \QQ_p[G]^-).$$
Put
$$a_{L/K,S}^-:=\iota_{w_1}^{-1}\left( \left(\frac{\sqrt{-1}}{\pi}\right)^r e^-\theta_{L/K,S}^\ast(1) \right)\in \overline \QQ[G]^-.$$
We define $\mathfrak{s}_{L/K,S}$ by
$$
\mathfrak{s}_{L/K,S}:=j(a_{L/K,S}^-) \cdot {\rm Log}_p^{(1)} \wedge \cdots \wedge {\rm Log}_p^{(r)}
\in \Hom_{\ZZ_p[G]^-}\left({\bigwedge}_{\ZZ_p[G]^-}^r U_{L_p}^{1,-},\overline \QQ_p[G]^-\right)^\#.
$$
%where $(\cdot)^\#$ denotes the involution $\sigma \mapsto \sigma^{-1}$.
Solomon proved that the image of $\mathfrak{s}_{L/K,S}$ lies in $\QQ_p[G]^-$, and that $\mathfrak{s}_{L/K,S}$ is independent of the choice of $j$ (see \cite[Prop. 2.16]{solomon}).

Now we prove the first assertion of the proposition.
By the definition of $\eta_{L/K,S}^{e^-}(1)$, it is sufficient to prove that the composition map
\begin{eqnarray}
{\bigwedge}_{\CC_p[G]^-}^r \CC_p Y_L(1) \stackrel{\sim}{\to} {\bigwedge}_{\CC_p[G]^-}^r (\CC_p  \otimes_\QQ L)^- \stackrel{\exp_p}{\stackrel{\sim}{\to}} \CC_p {\bigwedge}_{\ZZ_p[G]^-}^r U_{L_p}^{1,-} \stackrel{(-1)^{r}\mathfrak{s}_{L/K,S}}{\to} \CC_p[G]^- \nonumber
\end{eqnarray}
coincides with $ e^- \theta_{L/K,S}^\ast(1) \cdot w_1(1)^\ast \wedge \cdots \wedge w_r(1)^\ast$ in $\Hom_{\CC_p[G]^-}(\bigwedge_{\CC_p[G]^-}^r \CC_p Y_L(1), \CC_p[G]^-)^\#$, where the first map is induced by
$$\CC_p Y_L(1)\stackrel{(\ref{beta})}{\simeq}   \CC_p H_L(1)^+  \stackrel{(\ref{period map})}{\simeq}  (\CC_p\otimes_\QQ L)^-  ,$$
and the second by the $p$-adic exponential map. We denote by $\beta$ and $\alpha$ the maps induced by (\ref{beta}) and (\ref{period map}) respectively.

%For each $i$ with $1\leq i \leq r$, we see that the map $\log_p^{(i)}$ coincides with the composition map
%$$U_{L_p}^{1,-} \stackrel{\log_p}{\to} L_p =\QQ_p\otimes_\QQ L \stackrel{\iota_{w_i}}{\rightarrow} \QQ_p \otimes_\QQ \CC \simeq \QQ_p\otimes_\QQ \CC_p \to \CC_p,$$
%where the last map is defined by the product $a\otimes b \mapsto ab$. We denote this map simply by $\iota_i \circ \log_p$.
Noting that the equality $(-1)^{r}j(a_{L/K,S}^-)= (\pi \sqrt{-1})^{-r}e^-\theta_{L/K,S}^\ast(1)$ holds in $\CC_p[G]^-$ (via the isomorphism $\CC \simeq \CC_p$), we see that
$$(-1)^{r}\mathfrak{s}_{L/K,S}= (\pi \sqrt{-1})^{-r}e^-\theta_{L/K,S}^\ast(1) \cdot  {\rm Log}_p^{(1)}\wedge \cdots \wedge  {\rm Log}_p^{(r)}.$$
Hence, it is sufficient to prove that $ \iota_{w_i} \circ \alpha \circ \beta (w_i(1))= \pi \sqrt{-1} $, but this is straightforward to check by definition.
%Since the
%isomorphism $\alpha^{-1}$ is given by $a\otimes b \mapsto (a\cdot \iota(b))_{\iota : L \hookrightarrow
%\CC}$, we see by the definition of $w_i(1)$ that $w_i(1)^\ast \circ \beta^{-1} \circ \alpha^{-1} = (\pi \sqrt{-1})^{-1} \iota_{w_i}$, and
Thus we have proved the first assertion of the proposition.

To prove the second assertion we review the definition of $H_{L/K,S,n-1}$. In \cite[\S 2.3]{solomon}, Solomon defined a map $f_u \in \Hom_{\ZZ}(\co_{L^+,S}^\times,\ZZ/p^n)$ for any $u\in U_{L_p}^1$ as follows. For each $w\in S_p(L)$, we denote the Hilbert symbol
$$L_w^\times \times L_w^\times \to \mu_{p^n}; \ (x,y)\mapsto \frac{{\rm rec}_w(y) x^{1/p^n}}{x^{1/p^n}}$$
by $(x,y)_{w,n}$. The map $f_u$ is then defined by setting $f_u(a):=\sum_{w\in S_p(L)}\xi_1^\ast((a,u_w)_{w,n}),$
where $u_w$ is the $w$-component of $u \in U_{L_p}^1$ and $\xi_1^\ast$ is the isomorphism $\mu_{p^n} \stackrel{\sim}{\to} \ZZ/p^n$ sending $\xi_1$ to $1.$ (For the definition of $\xi_i$, see (\ref{def xi}).)

We define the ring isomorphism
$$\chi_{\rm cyc}^\# : \ZZ/p^n[G^+] \stackrel{\sim}{\to} \ZZ/p^n[G]^-; \ \sigma\mapsto 2^{-1} \sum_{\widetilde \sigma}\chi_{\rm cyc}(\widetilde \sigma)\widetilde \sigma^{-1},$$
where $\widetilde \sigma\in G$ runs over the lifts of $\sigma\in G^+$.
%Note that our $\kappa^\#$ is denoted by ${\bar \kappa}_n^\ast$ in \cite[\S 2.3]{solomon}.
The pairing $H_{L/K,S,n-1}$ is defined by
$$H_{L/K,S,n-1}(a,u_1\wedge\cdots\wedge u_r):=2^r \chi_{\rm cyc}^\#((\widetilde F_{u_1}\wedge\cdots\wedge \widetilde F_{u_r} )(a)),$$
where $\widetilde F_{u_i}:=\sum_{\sigma \in G^+} \widetilde f_{u_i}(\sigma (\cdot))\sigma^{-1}$, and $\widetilde f_{u_i} \in \Hom_{\ZZ_p}(\co_{L^+,S}^\times,\ZZ_p)$ is a lift of $f_{u_i}$.

We compare Solomon's pairing $H_{L/K,S,n-1}$ with the twisting map
\begin{multline*}
{\rm tw}_{0,1}: {\bigcap}_{\ZZ_p[G^+]}^r \ZZ_p\co_{L^+,S}^\times ={\bigcap}_{\ZZ_p[G]^+}^r H^1(\co_{L,S},\ZZ_p(1))^+ \\
\to {\bigcap}_{\ZZ/p^n[G]^-}^r H^1(\co_{L,S},\ZZ/p^n)^-=\Hom_{\ZZ_p[G]^-} \left({\bigwedge}_{\ZZ_p[G]^-}^r \Gamma_{L,S}^-,\ZZ/p^n[G]^- \right)^\#.
\end{multline*}
For each $i$ with $1\leq i \leq r$, we define $h_i: \Gamma_{L,S} \to \Hom_\ZZ(\co_{L,S}^\times,\ZZ/p^n)$ by
$$h_i(\gamma)(a):=\xi_i^\ast\left( \frac{\gamma a^{1/p^n}}{a^{1/p^n}} \right),$$
where $\xi_i^\ast: \mu_{p^n}\stackrel{\sim}{\to} \ZZ/p^n$ is defined by $\xi_i \mapsto 1$. Put $H_i:=\sum_{\sigma \in G^+}h_i(\sigma(\cdot))\sigma^{-1}$. One checks that
$${\rm tw}_{0,1}(a)(\gamma_1\wedge \cdots \wedge \gamma_r)=2^r \chi_{\rm cyc}^\# ((\widetilde H_1 \wedge \cdots \wedge\widetilde H_r)(a)),$$
where $\widetilde H_i\in \Hom_{\ZZ[G^+]}(\co_{L^+,S}^\times,\ZZ[G^+])$ is a lift of $H_i$. Note also that for all $a\in \co_{L^+,S}^\times$ and $u \in U_{L_p}^{1,-}$ one has
%$$\chi_{\rm cyc}(\tau_i)f_u(a)=h_{i}({\rm rec}_p(u))(a)$$
%holds for all $a\in \co_{L^+,S}^\times$ and $u \in U_{L_p}^{1,-}$. In fact,
%
\begin{multline*}
h_{i}({\rm rec}_p(u))(a)=\xi_i^\ast\left(\frac{{\rm rec}_p(u)a^{1/p^n}}{a^{1/p^n}}\right) = \xi_i^\ast\left(\prod_{w \in S_p(L)}\frac{{\rm rec}_w(u_w) a^{1/p^n}}{a^{1/p^n}}\right) \\
=\sum_{w\in S_p(L)}\xi_i^\ast((a,u_w)_{w,n}) = \sum_{w\in S_p(L)}\xi_1^\ast(\tau_i (a,u_w)_{w,n}) =\chi_{\rm cyc}(\tau_i)f_u(a). \end{multline*}

Hence we have
$${\rm rec}_p^\ast ( {\rm tw}_{0,1}(a))= \chi_{\rm cyc}(\tau_1\cdots \tau_r) H_{L/K,S,n-1}(a, \cdot) $$
in $\Hom_{\ZZ_p[G]^-}({\bigwedge}_{\ZZ_p[G]^-}^r U _{L_p}^{1,-},\ZZ/p^n[G]^-)^\#$, as required.
\end{proof}

\end{document}